\documentclass[11pt]{article}

\usepackage{hyperref}
\hypersetup{colorlinks = true, allcolors = {black}, linkcolor = {blue}, citecolor = {blue}}

\usepackage{amsmath}
\usepackage{amsfonts,bm,amssymb,tabularx,fancybox,ulem,mathtools,epsfig,epic,eepic,subfig,rotating}
\usepackage{epstopdf}
\usepackage{comment}
\usepackage[usenames,dvipsnames]{xcolor} 
\usepackage{makeidx,dsfont,wrapfig}
\usepackage[colorinlistoftodos]{todonotes}
\usepackage{lipsum}
\usepackage{fancyhdr} 
\usepackage{textcomp} 
\usepackage{graphicx}
\usepackage{cool} 
\usepackage{paralist} 
\usepackage{calrsfs}
\usepackage{multirow}
\usepackage{algorithm}
\usepackage[noend]{algpseudocode}

\usepackage{authblk}

\makeatletter
\def\BState{\State\hskip-\ALG@thistlm}
\makeatother

\usepackage{tikz}
\usetikzlibrary{shapes,arrows}
\usetikzlibrary{%
	decorations.pathreplacing,%
	decorations.pathmorphing%
}
\usepackage{cleveref} 

\newcommand{\qed}{\nobreak \ifvmode \relax \else
	\ifdim\lastskip<1.5em \hskip-\lastskip
	\hskip1.5em plus0em minus0.5em \fi \nobreak
	\vrule height0.75em width0.5em depth0.25em\fi}

\DeclareMathAlphabet{\pazocal}{OMS}{zplm}{m}{n}

\newcommand{\Ia}{\mathcal{I}}

\DeclarePairedDelimiter{\nm}{\lVert}{\rVert}
\DeclarePairedDelimiter{\abs}{\lvert}{\rvert}

\graphicspath{{./}{./im/}{./figures/}{./figures/surf/}{./figures/conc/}{./conc/}{./surf/}{./water/}{./pol/}}

\excludecomment{omitext}

\newcommand{\bv}{{\bf v}}

\newcommand{\bx}{{\bf x}}

\newcommand{\bn}{{\bf n}}
\newcommand{\bK}{{\bf K}}

\newcommand{\intl}{\int\limits}

\newcommand{\dl}{\ensuremath{\partial}}
\newcommand{\p}[2]{\ensuremath{\frac{\partial #1}{\partial #2}}}
\newcommand{\lb}{\ensuremath{\lambda}}
\newcommand{\etal}{{\it et~al}.\ }
\newcommand{\ie}{\textit{i}.\textit{e}.\ }

\newcommand{\Eq}[1]{Eq.~\eqref{#1}}
\newcommand{\Fig}[1]{Figure~\ref{#1}}
\newcommand{\Sect}[1]{\S~\ref{#1}}

\DeclareBoldMathCommand\balpha{{\alpha}}
\DeclareBoldMathCommand\btau{{\tau}}
\DeclareBoldMathCommand\bgamma{{\gamma}}
\DeclareBoldMathCommand\bbeta{{\beta}}
\DeclareBoldMathCommand\bgrad{{\nabla}}
\DeclareBoldMathCommand\bX{{\mathcal{X}}}

\newcommand{\pdt}[1]{\frac{\partial #1}{\partial t}}
\newcommand{\pdx}[1]{\frac{\partial #1}{\partial x}}

\renewcommand{\div}{\bgrad \cdot}

\newcommand{\R}{\mathbb{R}}

\newtheorem{theorem}{Theorem}[section]
\newtheorem{lemma}[theorem]{Lemma}

\newenvironment{proof}[1][Proof.]{\begin{trivlist}
		\item[\hskip \labelsep {\bfseries #1}]}{\end{trivlist}}

\newcommand{\ul}{\underline}      

\makeatletter
\let\ps@plain\ps@fancy 
\makeatother

\title{On the convergence analysis of a hybrid numerical method for multicomponent transport in porous media}
\date{\today}
\author[1]{Prabir Daripa\thanks{email:daripa@math.tamu.edu, ORCID ID: orcid.org/0000-0002-8771-0149}}  
\author[2]{Sourav Dutta\thanks{Author for correspondence  (email:sdutta.math@gmail.com, ORCID ID: orcid.org/0000-0002-7051-175X)}} 
\affil[1]{\scriptsize{Department of Mathematics, Texas A\&M University, 3368 TAMU, College Station, TX 77843, USA} }
\affil[2]{{\scriptsize U.S. Army Engineer Research and Development Center, Coastal and Hydraulics Laboratory (ERDC-CHL), Vicksburg, MS 39180, USA}}

\providecommand{\keywords}[1]{\textbf{\textit{Keywords---}} #1}
\providecommand{\msc}[1]{\textbf{\textit{AMS Subject Classifications---}} #1}

\begin{document}

\thispagestyle{fancy}
\maketitle

\begin{abstract}
	In this article, the convergence of a hybrid numerical method introduced in Daripa \& Dutta~({\sl J. Comput. Phys.}, 335:249-282, 2017) has been established. This method integrates a discontinuous finite element method with a modified method of characteristics (MMOC) in combination with finite difference (FD) procedures, and has been successfully applied to solve a coupled system of nonlinear equations that arises in multicomponent two-phase porous media flows. The present convergence analysis is focused on the MMOC-FD procedure for a nonlinear system of transport equations. For this purpose, an analogous single-component system of transport equations has been considered and some key ideas for possible extension to multicomponent systems have been briefly discussed. Error estimates have been obtained and these estimates have also been shown to be consistent with realistic numerical simulations of flows arising in enhanced oil recovery processes.
\end{abstract}


\smallskip

\msc{65M12, 65M25, 65M06, 76S05}

\section{Introduction}
In Daripa and Dutta \cite{DD2017} we developed a hybrid numerical method for solving a coupled system of elliptic and transport equations that arises in modeling multicomponent, multiphase porous media flow in the context of chemical Enhanced Oil Recovery (EOR) by Surfactant-Polymer-flooding (SP-flooding). The mathematical model involves two immiscible fluids (water and oil) with two components (polymer and surfactant) present in one of the fluids (aqueous phase). The hybrid method is derived from a non-traditional discontinuous finite element method and a time-implicit finite difference method based on the \ul{M}odified \ul{M}ethod \ul{O}f \ul{C}haracteristics (MMOC). Numerical results obtained with this method for a variety of initial data in rectilinear and radial geometries are in excellent agreement qualitatively with physics based expectation and converge under mesh refinement. In some cases where exact solutions are available, numerical results are in excellent agreement with the exact ones as well. 

In this paper, we present a convergence analysis of the numerical method. For the analysis, we consider a reduced system of equations in one spatial dimension involving only one component (polymer). This reduced system models chemical \ul{E}nhanced \ul{O}il \ul{R}ecovery (EOR) by polymer flooding~\cite{DGLM1988-1,DGLM1988-2} in one spatial dimension. Previous studies \cite{DR1982,D1983} on the convergence analysis of the MMOC-based methods have been focused on two-phase flow without components and the present study builds on that work by adding a component. This complicates the analysis by introducing additional terms due to the coupling of the coefficients involved in the transport equations. The potential and some of the challenges in extending the convergence analysis presented here to the original problem involving two components and in two dimensions have been briefly discussed, however a complete analysis of the original problem remains open. For the purpose of validating the error estimates obtained from the analysis, we carry out numerical simulations of polymer flooding and compute $L^2$ and $L^\infty$ error norms for the numerical solutions. From a practical standpoint, the primary focus of secondary and tertiary (EOR) recovery processes is in the interior of the domain (reservoir) and not at its boundaries. 
This is accounted for in the form of some simplifying regularity assumptions in the analysis, that facilitate the study of the convergence behavior of the numerical method and that are similar to the assumptions made in the analytical study of many other numerical methods~\cite{D1983,EW1980,R1985} for multiphase transport problems. 

To this end, it is worth citing some methods, without being exhaustive by any means, on numerically solving similar systems of partial differential equations arising in porous media flows. These methods can be broadly categorized into two classes: purely Eulerian and Eulerian-Lagrangian. Some of the methods which fall under Eulerian class are locally mass conservative finite volume methods~\cite{D1993,CPV2014} and finite element based methods such as control volume, continuous Galerkin~\cite{EHM2003}, discontinuous Galerkin~\cite{CCSS2001,KP2016} and mixed finite element~\cite{NSM2004} which have high order accuracy and have been applied for numerical simulations of porous media flows. Some of the methods which can be grouped under Eulerian-Lagrangian class are front-tracking methods~\cite{DGLM1988-1,DGLM1988-2} and MMOC based methods~\cite{DR1982,D1983}. There are many variants of these methods such as modified method of characteristics with adjusted advection~\cite{DHP1999}, the Eulerian-Lagrangian localized adjoint method~\cite{CRHE1990}, and the characteristic mixed finite element method~\cite{AW1995}, to name a few. Error estimates and convergence analyses of most of these methods have been carried out (see \cite{DR1982,CCSS2001,AW1995,W2000,WW2007,BK2012}) for two-phase flow systems without components. The present analysis is on a system with components.

The analysis involves estimation of errors introduced by the finite difference discretization of the derivatives, by the linear interpolation applied to compute solutions at non-nodal points where the characteristic curves intersect the computational grid and by the linearization of the coefficients. The coefficient functions and the auxiliary variables in the two transport equations depend on both the wetting phase saturation and the component concentration. This coupling creates an additional challenge for the analysis of the multicomponent system. The finite difference discretization errors are estimated using multi-variable Taylor Series. The errors due to the linear interpolation are estimated using the Peano kernel theorem \cite{DR1982} and the errors due to quasi-linear approximation of the nonlinear coefficients are estimated using various inequalities including the Cauchy-Schwarz and generalized arithmetic mean-geometric mean (AM-GM) inequalities. The aforementioned estimates are obtained in terms of the inner products of the relevant error variables. These are used to reformulate the transport equations in a way that allows us to estimate the discrete $L^2$ errors in the aqueous phase saturation and the component concentration. Taking into account the temporal discretization, a discrete Gr\"{o}nwall type inequality is finally used to obtain the desired estimates. 

The rest of the paper is laid out as follows. In \cref{sec:problem}, we discuss the governing equations for incompressible, multicomponent, immiscible two-phase flow of fluids through porous media. In \cref{sec:method}, we present the numerical method: the computational grids, the non-traditional discontinuous finite element method, the MMOC based finite difference scheme for the transport equations and the computational algorithm. In \cref{sec:analysis}, we present the convergence analysis of the method and the relevant error estimates. We present the numerical results and compare them with the theoretical error estimates in \cref{sec:results}. Finally \cref{sec:conclusions} contains concluding remarks.

\section{Background}
\subsection{Model}\label{sec:problem}
In \cite{DD2017}, a system of equations governing two-phase, two-component (SP-flooding) flow through porous media has been presented. With the possibility of some potential overlap, here we present the model for a single component flow (polymer flooding) system which will be later used for the analysis of the numerical method. 

Let $\Omega \subset \R^2$ represent a porous medium with boundary $\dl \Omega$. The incompressible and immiscible flow of the wetting phase (water or an aqueous solution of polymer and/or surfactant) and the non-wetting phase (oil) is described by a combination of the multiphase extension of Darcy's law (see \cite{MM1936}) for each phase and transport equations for each component. Let $s_j$ denote the saturation (volume fraction), $\bv_j$ denote the velocity, $p_j$ denote the phase pressure and $q_j$ denote the volumetric injection/production rate of phase $j$ where $j = o$ and $j= a$ denote the non-wetting and the wetting phases respectively. We recall from Daripa \& Dutta~\cite{DD2017} the phase transport equations 
\begin{align}
\phi \p {s_j}{t} + \bgrad \cdot \bv_j &= q_j,  \quad (\bx,t) \in \Omega \times (0,T], \quad j = a,\,o,  \label{paper2-sat}
\end{align}
and the equation for conservation of mass of any component dissolved in the aqueous phase 
\begin{align}\label{paper2-conc}
\phi \p {(cs_a)}{t} + \div (c\bv_a) = c^i q_a^+ - c q_a^-,  \quad (\bx,t) \in \Omega \times (0,T] , 
\end{align}
where $c$ is the concentration (volume fraction in the aqueous phase) of the dissolved component, $c^i$ is the concentration of the component in the injected fluid, $q_a^+ = \max(q_a,0) $ and $q_a^- = \max(-q_a,0)$. The inherent assumption in this model is that the component is passively advected with negligible diffusion and adsorption. Using conservation of momentum of each phase, the phase velocity $v_j$ is given by the Darcy-Muskat law
\begin{align}
\bv_j &= - \bK(\bx) \lb_j \bgrad p_j, \quad \bx \in \Omega, \quad j=a,\, o. \label{paper2-darcy}
\end{align}
Here $\phi$ is the porosity (taken to be constant in the numerical experiments in this study), $\bK(\bx)$ is the absolute permeability tensor of the porous medium, $\lb = \lb_a(s,c) + \lb_o(s,c)$ is the total mobility where $\lb_j = k_{rj}/\mu_j$ is the phase mobility, $k_{rj}$ is the relative permeability and $\mu_j$ is the viscosity of phase $j$. In addition to the above, the capillary pressure ($p_c$) is defined by
\begin{align}\label{paper2-cap}
p_c = p_o - p_a.
\end{align}
Since the porous medium is initially saturated with the two phases, we have
\begin{align}\label{paper2-saturated}
\sum_{j=o,a} s_j = 1.
\end{align}
The combination of the above equations produces a system of strongly coupled nonlinear equations which can be potentially degenerate. In order to avoid this difficulty, we reformulate the problem by using a fictitious pressure ($p$), to be called the global pressure below, for incompressible, immiscible two-phase flows with a single component (see \cite{DD2017}) defined by
\begin{align}\label{paper2-gblpnew}
p = \frac{1}{2} (p_o+p_a) &+ \frac{1}{2} \int^{s}_{s_c}\left(\hat{\lb}_o(\zeta,c) - \hat{\lb}_a(\zeta,c)\right)\frac{dp_c}{d\zeta}(\zeta)d\zeta \nonumber\\
 \qquad &- \frac{1}{2}\int \left(\int^{s}_{s_c}\frac{\dl}{\dl c}\left(\hat{\lb}_o(\zeta,c) - \hat{\lb}_a(\zeta,c)\right)\frac{dp_c}{d\zeta}(\zeta)d\zeta\right) \left(\p{c}{x}dx +\p{c}{y}dy\right),
\end{align}
where $\hat{\lb}_j = \lb_j/\lb $ for $ j = a,o$ and $s_c$ is the value of the aqueous phase saturation for which $p_c(s_c) = 0$. The global pressure is well defined for all values of $s_a$ in $[s_{ra},1-s_{ro}]$ where $s_{ra}$ (resp. $s_{ro}$) is the residual saturation of the wetting phase (resp. non-wetting phase). The global pressure was introduced by Chavent and Jaffr\'{e} \cite{CJ1986} for immiscible, incompressible two-phase flow without any components and was later revisited by several others (see for instance \cite{AJK2014}). If we write $s_a = s$, an equivalent formulation of the problem is obtained in terms of the primary variables $(p,s,c)$ as
\begin{subequations}
 \begin{align}
&- \div ( \bK(\bx)\lb\bgrad p )  = q,  &\bx \in \Omega, \ t \in (0,T] , \label{paper2-glbfinal}\\
	&\phi \p{s}{t} + \p{f}{s} \bv \cdot \bgrad s - \div (\bm{D} \bgrad s)= G_s - \p{f}{c} \bv \cdot \bgrad c, &\bx \in \Omega, \ t \in (0,T], \label{paper2-sateq}\\
&\phi \p{c}{t} + \left(\frac{f}{s} \bv - \frac{\bm{D}}{s} \bgrad s \right)\cdot \bgrad c  = G_c, &\bx \in \Omega, \ t \in (0,T],\label{paper2-conceq}
\end{align} 
\end{subequations}
\noindent where $\bv:= \bv_a + \bv_o = -\bK(\bx)\lb\bgrad p$ is the total velocity and $\bm{D}(s,c) = - \bK(\bx)\lambda_o(s)f(s,c)\frac{d p_c(s)}{d s}$ is the capillary pressure induced diffusion coefficient. Also, the total external flow rate $q = (q_a + q_o)$ is an appropriate source term for the pressure equation which denotes net volume of external fluid containing the non-wetting phase ($q_o$) and the wetting phase ($q_a$), injected per unit volume per unit time and the source terms for the transport equations are modeled by
\begin{align}\label{paper2-eq:compat1}
(G_s, G_c) = \begin{cases} \left((1-f)q, \dfrac{(c^i - c)q}{s} \right), \, &q \geq 0 \\ (0,0) \, , \, &q<0. \end{cases}
\end{align}	

From the modeling perspective, the above assumptions signify that oil is never injected and the fluid mixture obtained at the production well is proportional to the resident fluid at the point. In \cref{paper2-sat,paper2-conc,paper2-glbfinal,paper2-sateq,paper2-conceq}, it is inherently assumed that relative to the computational domain, the source terms are ideally a collection of unit impulses. For instance, in \cite{EW1980} the authors model source terms as point sources. However, motivated by physics, they assume $q_j (j=a,o)$ to be smooth which corresponds to smoothly distributed sources and sinks, and also that $|q_j(\mathbf{x},t)| \leq K$ for some constant $K$. In \cite{D1983,R1985}, the source terms are also modeled in a similar fashion whereas, for the analysis they require the source terms to be bounded in space and smooth in time. These assumptions improve the regularity of the exact solution and are justified from the standpoint of modeling the kind of tertiary recovery problems that are of interest here. Such modeling practices are summarized in \cite{E1991}, \cite{CHM2006} (Chap 2 and 13) and the references cited therein. Following the same modeling philosophy, for the analysis and the numerical solution of the flow and transport system (\cref{paper2-glbfinal,paper2-sateq,paper2-conceq}) we assume that the external flow rate functions $q_a,q_o$ are smoothly distributed \cite{DFP1997}.

The following initial and boundary conditions are prescribed to complete the problem description.
\begin{subequations}
\begin{align}
 &\quad \forall \bx \in \Omega : \qquad s(\bx,0) = s_0(\bx) \quad \& \quad c(\bx,0) = c_0(\bx),    \label{paper2-ic}\\
 &\forall (\bx,t) \in \dl \Omega \times (0,T] : \quad \bgrad s \cdot\hat{\bn} = 0, \quad \bgrad c \cdot\hat{\bn} = 0 \quad \& \quad  \bv_j\cdot\hat{\bn} = 0 \,  (j = a,o), \label{paper2-bc}
\end{align}
\end{subequations}
where $\hat{\bn}$ denotes the outward unit normal to $\dl \Omega$. The no-flow boundary conditions in \cref{paper2-bc} are equivalent to requiring an impermeable boundary of the reservoir and impose a compatibility condition to the source terms in \cref{paper2-eq:compat1} as given by $\int_{\Omega} q dx = 0$. The phase velocity boundary conditions as stated in \cref{paper2-bc} lead to a no-flow condition for the total velocity given by, $\bv\cdot\hat{\bn} = (\bv_a + \bv_o)\cdot\hat{\bn} = 0$. This, in turn, translates to an effective boundary condition for the global pressure \cref{paper2-glbfinal} as $\bK(\bx)\lb\bgrad p \cdot \hat{\bn} = 0$. Clearly, \cref{paper2-glbfinal} and the aforementioned homogeneous Neumann boundary condition will determine $p$ only to within an additive constant. Thus, a normalizing constraint $\int_\Omega p \,dx = 0$ is imposed in order to ensure uniqueness \cite{E1991}. Moreover, as the total velocity $\bv$, the phase velocities $\bv_a$, $\bv_o$ and other quantities of interest depend on the gradient of pressure, an arbitrary normalization constraint on the pressure does not affect the computation of the transport variables.

Several models of relative permeability, $k_{rj}$ and capillary pressure are available in the literature (see \cite{C1986,vG1980}). 
For the numerical simulations presented in this study, we use the following modification to the van Genuchten model made by Parker \etal (\cite{PLK1987}),
\begin{subequations}
\begin{align}
 k_{ra}(s)&= s_e^{1/2}\left(1-(1-s_e^{1/m})^m\right)^2,\label{paper2-eq:relperm1}\\
 k_{ro}(s)&= (1-s_e)^{1/2}\left(1-s_e^{1/m}\right)^{2m},\label{paper2-eq:relperm2}\\
 p_c(s)&= \frac{1}{\alpha_0}\left(s_e^{-1/m}-1 \right)^{1-m} \label{paper2-eq:cap},
\end{align}
\end{subequations}
where $s_e = (s-s_{ra})/(1-s_{ra})$ is the effective saturation. The values of the parameters $m$ and $\alpha_0$ in the above model are known to depend on the interfacial tension, denoted by $\sigma0$, between the non-wetting and the wetting phases. In our study below we take $m=2/3$ and $\alpha_0 = 0.125$ (see \cite{GLGV2010}). It is also assumed that $s_e>0$ as $s_{ra} < s_0^{\sigma0} < s$ which means that the wetting phase saturation $s$ is bounded below by the initial resident saturation of the aqueous phase $s_0^{\sigma0}$ which is greater than the residual saturation of the wetting phase. Alternatively Corey-type imbibition relations can also be used (see \cite{C1986}). These models assume that $p_c$ and $k_{rj} (j=a,o)$ are nonlinear functions of only the wetting phase saturation, $s$, which is a valid assumption for polymer flooding. The aqueous phase viscosity is modeled by a linear function of polymer concentration,
\begin{align}\label{paper2-eq:vis}
    \mu_a = \mu_w(1+ \gamma c),
\end{align}
where $\mu_w$ is the viscosity of pure water and the coefficient $\gamma$ characterizes the particular polymer.

\subsection{Numerical scheme}\label{sec:method}
The system of coupled transport equations given by \cref{paper2-sateq,paper2-conceq} is solved using a combination of the MMOC and an implicit time finite difference scheme. For the computational grid, we partition the domain $\Omega$ into rectangular cells. Given positive integers $I,J \in \mathcal{Z}^+$, set $\Delta x = (x_{max}-x_{min})/I = 1/I $ and $\Delta y = (y_{max} - y_{min})/J =1/J$. We define a uniform Cartesian grid ${(x_{i},y_{j})} = {(i\Delta x,j\Delta y)} $ for $i = 0, . . ., I $ and $j = 0, . . ., J$. Each $(x_{i},y_{j})$ is called a grid point. For the case $i = 0, I$ or $j = 0, J$, a grid point is called a boundary point, otherwise it is called an interior point. In general, the grid size is defined as $h = max(\Delta x,\Delta y) > 0$. However, in this paper we use an uniform spatial grid: $\Delta x = \Delta y = h=1/N$.

The elliptic flow equation~\eqref{paper2-glbfinal} for global pressure is solved using a discontinuous finite element method on a non-body-fitted grid which is constructed in the following way. We introduce uniform triangulations inside the grid generated for the transport equations~\eqref{paper2-sateq} and \eqref{paper2-conceq}. This means every rectangular region $[x_{i},x_{i+1}]\times[y_{j},y_{j+1}]$ is cut into two pieces of right triangular regions: one is bounded by $x = x_{i}, y = y_{j} $ and $y=\frac{y_{j+1}-y_{j}}{x_{i}-x_{i+1}}(x-x_{i+1})+y_{j}$, the other is bounded by $x = x_{i+1}, y = y_{j+1}$ and $y=\frac{y_{j+1}-y_{j}}{x_{i}-x_{i+1}}(x-x_{i+1})+y_{j}$. Collecting all those triangular regions, also called elements, we obtain a uniform triangulation, $L^h =\{\kappa \vert \kappa \text{ is a triangular element}\} $. We may also choose the hypotenuse to be  $y=\frac{y_{j+1}-y_{j}}{x_{i+1}-x_{i}}(x-x_{i})+y_{j}$, and get another uniform triangulation from the same Cartesian grid. There is no conceptual difference on these two triangulations for our method.

\subsubsection{Pressure equation}
The elliptic equation describing the evolution of global pressure is given by
\begin{subequations} 
\begin{align}
- \bgrad\cdot \left( \bK(\bx)\lambda\bgrad p \right) &=\tilde{q}, \quad   \ \bx\in\Omega\backslash\Sigma,\label{eq:paper2-ell-1}\\
\left( \bK(\bx)\lambda\bgrad p \right) \cdot \hat{\bn} &= 0, \quad \ \bx\in\partial\Omega,   \label{eq:paper2-ell-2}
\end{align}
\end{subequations}
where $\tilde{q} = q_a + q_o$ and $\Sigma$ denotes the union of the interfaces that separate $\Omega$ into several subdomains.  However, for simplicity of exposition we assume here that we have only two separated subdomains, $\Omega^{+}$ and $\Omega^{-}$ separated by an interface $\Sigma$ (see \Cref{fig:paper2-num-method}) which, as we will see later, is also the initial configuration of the quarter five-spot domain for all of our numerical simulations. 

\begin{figure}[h!]
	\centering
\begin{tikzpicture}[scale=1,
interface/.style={
	postaction={draw,decorate,decoration={border,angle=-45,
			amplitude=0.3cm,segment length=2mm}}},]
\draw[line width=0.2mm] (0,5) -- (5,5);
\draw[line width=0.2mm] (5,0) -- (5,5);
\draw[line width=0.2mm] (0,0)-- (0,5);
\draw[line width=0.2mm] (0,0)--(5,0);
\draw[line width=0.1mm] (0,1) -- (1,0);
\draw[line width=0.1mm] (0,2) -- (2,0);
\draw[line width=0.1mm] (0,3) -- (3,0);
\draw[line width=0.1mm] (0,4) -- (4,0);
\draw[line width=0.1mm] (0,5) -- (5,0);
\draw[line width=0.1mm] (1,5) -- (5,1);
\draw[line width=0.1mm] (2,5) -- (5,2);
\draw[line width=0.1mm] (3,5) -- (5,3);
\draw[line width=0.1mm] (4,5) -- (5,4);

\fill[blue!25!,opacity=.3] (0,0) rectangle (5,5);
\draw [line width=0.1mm, gray] (0,0) grid (5,5);

\draw[blue,thick] (0:2.2) arc (0:90:2.2) ;	
\normalsize
\draw	(1,-0.2) node [text=black,below] { $(0,0)-$  Source}
(4,5.2) node [text=black,above] { $(1,1)-$ Sink}
(3.5,1.5) node [text=black] { $\Omega^-$}
(.7,1.5) node [text=black] { $\Omega^+$}
(2.4,.6) node [text=black] { $\Sigma$}
(2.3,3.3) node [text=black,thick] { $\kappa$};	

\draw[gray,line width=.5pt,interface](0,0)--(5,0) (5,0)--(5,5) (5,5)--(0,5) (0,5)--(0,0);

\foreach \Point in {(0,0),(5,5)}{\fill[red] \Point circle[radius=4pt];}

\draw[line width=0.8mm, OliveGreen] (2,3) -- (2,4);
\draw[line width=0.8mm, OliveGreen] (2,4) -- (3,3);
\draw[line width=0.8mm, OliveGreen] (3,3) -- (2,3);
\end{tikzpicture}
\caption{Initial configuration of the solution domain with the (blue) arc representing the initial position of the discontinuity}
\label{fig:paper2-num-method}
\end{figure}
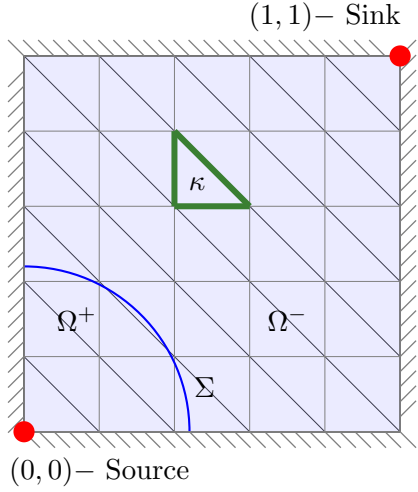

The following kinematic condition holds at the interface $\Sigma$.
\begin{equation}\label{eq:paper2-elljump}
\left[\bK(\bx)\lb\bgrad p\cdot\hat{\bn}\right]_\Sigma = 0,
\end{equation}
where $\hat{\bn}$ is the outward unit normal which points from $\Omega^-$ to $\Omega^+$ and $[ \,]$ denotes a jump. We assume the boundary $\partial \Omega$ and the interface $\Sigma$ to be Lipschitz continuous. Hence a unit normal vector, $\hat{\bn}$ can be defined a.e.~on $\Sigma$. This problem is solved using a non-traditional finite element formulation (see  \cite{HWW2010})  which is second order accurate in the $L^{\infty}$ norm for matrix coefficient elliptic equations with discontinuities across the interfaces. The weak formulation  of \cref{eq:paper2-ell-1,eq:paper2-ell-2} in the usual Sobolev spaces $H^1(\Omega)$ with $\psi \in H^1(\Omega)$ is given by
\begin{equation}\label{eq:paper2-weakform}
\int\limits_{\Omega^+} \bK\lambda\bgrad p\bgrad \psi + \int\limits_{\Omega^-} \bK\lambda\bgrad p\bgrad \psi -\int\limits_{\partial\Omega} \bK\lambda\psi\bgrad p \cdot \hat{\bn} = \int\limits_\Omega \tilde{q} \psi. 
\end{equation}

The elements, $\kappa$ of triangulation, $L^h$, are classified into regular cells and interface cells. We call $\kappa$ a regular cell if its vertices are in the same subdomain and an interface cell  when its vertices belong to different subdomains. For an interface cell, $\kappa = \kappa^+ \cup \kappa^-$ where $\kappa^+$ and $\kappa^-$ are separated by a line segment $\Sigma^h_k$, obtained by joining the two points where the interface $\Sigma$ intersects the sides or the vertices of that interface cell.  A set of grid functions, $H^{1,h} = \{\omega^h\,|\, \omega^h = \omega_{i,j} \,;\, 0 \leq i \leq I, 0\leq j\leq J\}$ are defined on the grid points of the mesh $L^h$. An extension operator $U^h: H^{1,h} \rightarrow H^1(\kappa)$ is constructed as follows. For any $\phi^h \in H^{1,h}$, $U^h(\phi^h)$ is a piecewise linear function and matches $\phi^h$ on the grid points. In a regular cell, it is a linear function that interpolates the values of $\phi^h$ at the grid points. In an interface cell, it consists of two pieces of linear functions, each defined on $\kappa^+$ and $\kappa^-$. The location of the discontinuity of the extended function $U^h(\phi^h)$ in an interface cell is on the line segment $ \Sigma^h_k$. Hence an interface jump condition on the pressure, $p$, if there is one, can be imposed on the two end points of this line segment at $\{ \partial \kappa \} \cap \{ \Sigma_k^h \}$ while the interface jump condition, \cref{eq:paper2-elljump}, is imposed at the middle point of $\Sigma_k^h$. For the construction of such extension operators for discontinuous coefficient elliptic equations, see~\cite{HWW2010,LLW2003}. The extension operators for the pressure equation, given by \cref{eq:paper2-ell-1,eq:paper2-ell-2}, have been explicitly constructed in \cite{DD2017}.\\
Using the extension functions as discussed above (also see \cite{DD2017}), the discrete version of the weak formulation \cref{eq:paper2-weakform} can be reformulated to finding a discrete function $\phi^h \in H^{1,h}$ such that
\begin{equation}\label{eq:paper2-final-weak}
\begin{split}
&\sum_{K \in L^h} \left(\,\int\limits_{K^+} \bK \lb \ \bgrad U^h(\phi^h) \bgrad U^h(\psi^h) + \intl_{K^-} {\bK} \lb \ \bgrad U^h(\phi^h) \bgrad U^h(\psi^h)\right) \\
& \quad - \sum_{K \in L^h} \int\limits_{\partial K} {\bK} \lb \ U^h(\psi^h)\bgrad U^h(\phi^h) \cdot \hat{\bn} = \sum_{K \in L^h}\left(\, \intl_{K^+} \tilde{q} \,U^h(\psi^h) + \intl_{K^-} \tilde{q} \,U^h(\psi^h)\right), \quad \forall \psi^h \in H^{1,h}.
\end{split}
\end{equation}
It can be shown that if $\bK(\bx)$ is positive definite, then the matrix obtained for the linear system of the discretized weak form, \cref{eq:paper2-final-weak}, is also positive definite (see Theorem 3.2 in \cite{HWW2010}) and is therefore invertible.

\subsubsection{Transport equations}
The transport equations \eqref{paper2-sateq} and \eqref{paper2-conceq} are solved using a combination of a finite difference method with the \ul{M}odified \ul{M}ethod \ul{O}f \ul{C}haracteristics (MMOC). At first we rewrite \cref{paper2-sateq,paper2-conceq} as 
\begin{subequations}
\begin{align}
&\phi \p{s}{t} + \p{f}{s} \textbf{v}\cdot \bgrad s - \div (\bm{D} \bgrad s)= g_s - \p{f}{c} \textbf{v}\cdot \bgrad c, \label{paper2-satnum}\\
&\phi \p{c}{t} + \left(\frac{f}{s} \textbf{v} - \frac{\bm{D}}{s} \bgrad s \right)\cdot \bgrad c  + c g = g_c, \label{paper2-satpol}
\end{align}
\end{subequations}
where $\bm{D}(s,c) = -\bK(\bx)\lambda_o(s)f(s,c)\frac{d p_c(s)}{d s}$, 
\begin{align}
    (g_s, g, g_c) = \begin{cases} \left((1-f)q, \dfrac{q}{s}, c^i \dfrac{q}{s}\right),   &q \geq 0  \\(0, 0, 0)\, , \,  &q<0\end{cases}.
\end{align}

In \cref{paper2-satnum} we replace the advection term $\phi \p{s}{t}+ \p{f}{s} \bv \cdot \bgrad s$ by a derivative along its characteristic direction 
in the following way 
\begin{align}\label{eq:paper2-char-deriv1}
\frac{\partial}{\partial \tau_s} = \frac{1}{\psi_s}\left(\phi \frac{\partial}{\partial t} + \p{f}{s} \textbf{v} \cdot \bgrad  \right),
\end{align}
where $\tau_s$ is used to parametrize the characteristics. Here $\psi_s$ is a suitable normalization that simplifies the numerical discretization of the characteristic derivative and is defined by
\begin{align}\label{eq:paper2-char-norm1}
\psi_s = \left[\phi^2 + \left(\p{f}{s}\right)^2 |\textbf{v}|^2\right]^{1/2}.
\end{align}
Then \cref{paper2-satnum} is equivalently written in the form
\begin{equation}\label{eq:paper2-char-approx1}
\psi_s \p{s}{\tau_s} - \div(\bm{D} \bgrad s) = g_s - \p{f}{c} \bv \cdot \bgrad c . 
\end{equation}

\begin{figure}[h!]
\centering
\begin{tikzpicture}[scale=0.7]
\draw [latex-latex] (0,-.5)-- (0,5.5) node [above] {$t$};
\draw[latex-latex] (-.5,0)--(5.5,0) node [right] {$\bx$};
\draw[dashed]	(0,4) --(4,4)
	(0,2) -- (2,2)
	(2,0) -- (2,2)
	(4,0) --(4,4);
\draw [dotted, gray] (0,0) grid (5,5);
\foreach \Point in {(0,2),(0,4),(2,0),(4,0)}{\draw \Point circle[radius=2pt];}
\foreach \Point in {(0,2),(0,4),(2,0),(4,0)}{\fill \Point circle[radius=2pt];}
\draw[thick,solid,blue] (1.5,1.5) -- (4.5,4.5);
\foreach \Point in {(2,2),(4,4)}{\fill[red] \Point circle[radius=3pt];}
\draw	(2,0) node [text=black,below] {$\bar{\bx}_{ij}$}
	(4,0) node [text=black,below] {$\bx_{ij}$}
	(0,4) node [text=black, left] {$t^{n+1}$}
	(0,2) node [text=black, left] {$t^{n}$}
	(2,1.7) node [text=black, right] {$p_2$}
	(4,3.7) node [text=black, right] {$p_1$}; 	
\end{tikzpicture}
\caption{Discrete approximation of the characteristic curve from $\bar{\bx}_{ij}$ to $\bx_{ij}$ in 1D}
\label{fig:paper2-char-method}
\end{figure}
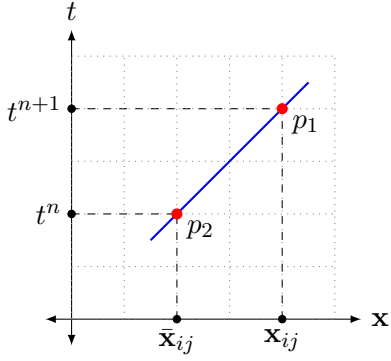

For computation, we use the spatial grid described in the beginning of \Sect{sec:method} and the time interval $[0,T]$ is uniformly divided into $L$ sub-intervals of length $\Delta t$ such that $t^n = n\Delta t$ and $T = L\Delta t$. We denote the grid values of the variables by $w_{ij}^n = w(\bx_{ij},t^n)$ where $\bx_{ij} = \bx(ih,jh)$. Consider that the solution is known at some time $t^n$ and the solution at a subsequent time $t^{n+1}$ needs to be computed. Then starting from any point $(\bx_{ij},t^{n+1})$ we trace backward along the characteristics to a point $(\bar{\bx}_{ij},t^n)$ where the solution is already known.  As shown in \Cref{fig:paper2-char-method}, the points $p_1=(\bx_{ij},t^{n+1})$ and $p_2=(\bar{\bx}_{ij}, t^n)$ lie on the same characteristic curve.
From the equation of the characteristic curves given by
\begin{align*}
 \frac{d\bx}{d\tau_s} =\frac{1}{\phi}\p{f}{s}\bv,  
\end{align*} 
we use numerical discretization to obtain an approximate value of $\bar{\bx}_{ij}$ in the following way
\begin{align*}
\bar{\bx}_{ij} = \bx_{ij} - \p{f}{s}({s}_{ij}^n,{c}_{ij}^n) \bv_{ij}^n \Delta t/\phi.
\end{align*}
Using the above equation, the derivative in the characteristic direction, defined by \cref{eq:paper2-char-deriv1}, is approximated by
\begin{align*}
\psi_s \p{s}{\tau_s} &\approx \psi_s \frac{s(\bx_{ij},t^{n+1})-s(\bar{\bx}_{ij},t^{n})}{\left[\abs{\bx_{ij}-\bar{\bx}_{ij}}^2+(\Delta t)^2\right]^{1/2}}= \phi \frac{s_{ij}^{n+1}-\bar{s}_{ij}^n}{\Delta t} \quad \text{( see \Sect{sec:analysis} )}.
\end{align*}
This leads to the following implicit-time finite difference formulation for \cref{eq:paper2-char-approx1}\\
\begin{align}\label{eq:paper2-sat}
\phi \frac{s_{ij}^{n+1}-\bar{s}_{ij}^n}{\Delta t}  - \bgrad _h(\bar{\bm{D}}\bgrad_h s)_{ij}^{n+1} &= (g_s)_{ij} -  \left(\p{f}{c}\right)_{ij}^{n} \left( \textbf{v}_{ij}^n  \cdot \bgrad_h c_{ij}^n\right),
\end{align}
where
\begin{align*}
\bar{s}_{ij}^n &= s(\bar{\bx}_{ij},t^{n})  \qquad \& \qquad \bar{\bm{D}}_{ij}^n = \bm{D}(\bar{s}_{ij}^n, c_{ij}^n), \\
\bgrad _h(\bar{\bm{D}}\nabla_h s)_{ij}^{n+1} &= \bar{\bm{D}}_{i+1/2,j}\frac{s_{i+1,j}^{n+1}-s_{i,j}^{n+1}}{\Delta x ^2} - \bar{\bm{D}}_{i-1/2,j}\frac{s_{i,j}^{n+1}-s_{i-1,j}^{n+1}}{\Delta x ^2}\\
&+\bar{\bm{D}}_{i,j+1/2}\frac{s_{i,j+1}^{n+1}-s_{i,j}^{n+1}}{\Delta y ^2}-\bar{\bm{D}}_{i,j-1/2}\frac{s_{i,j}^{n+1}-s_{i,j-1}^{n+1}}{\Delta y ^2},\\
\bar{\bm{D}}_{i\pm1/2,j}&= \frac{\bm{D}(\bar{s}_{i\pm1,j}^n,c_{i\pm1,j}^n) + \bm{D}(\bar{s}_{i,j}^n,c_{i,j}^n)}{2},\\
\bar{\bm{D}}_{i,j\pm 1/2}&= \frac{\bm{D}(\bar{s}_{i,j\pm1}^n,c_{i,j\pm1}^n) + \bm{D}(\bar{s}_{i,j}^n,c_{i,j}^n)}{2}.
\end{align*}
Following the same procedure as before we define the following equations, analogous to \cref{eq:paper2-char-deriv1,eq:paper2-char-norm1}, for the concentration equation \cref{paper2-satpol}.
\begin{align*}
\frac{\partial}{\partial \tau_c} &= \frac{1}{\psi_c}\left(\phi \frac{\partial}{\partial t} + \frac{f}{s} \textbf{v} \cdot \nabla - \frac{\bm{D}}{s}\nabla s\cdot \nabla  \right),\\
\psi_c &= \left[\phi^2 + \left(\frac{f}{s}\right)^2 \abs{\textbf{v}}^2 - \left(\frac{\bm{D}}{s}\right)^2\abs{\nabla s}^2\right]^{1/2}.
\end{align*}
The advection term $\phi \p{c}{t} + \frac{f}{s} \bv \cdot \bgrad c - \frac{\bm{D}}{s} \bgrad s \cdot \bgrad c$ is replaced by the derivative in the characteristic direction $\tau_c$ given by $\psi_c \frac{\partial}{\partial \tau_c}$ where 
\begin{align*}
\psi_c \p{c}{\tau_c} &\approx \psi_c \frac{c(\bx_{ij},t^{n+1})-c(\bar{\bx}^c_{ij},t^{n})}{\left[\abs{\bx_{ij}-\bar{\bx}_{ij}}^2+(\Delta t)^2\right]^{1/2}}= \phi \frac{c_{ij}^{n+1}-\bar{c}_{ij}^n}{\Delta t}. 
\end{align*}
Here, $\bar{c}_{ij}^n = c(\bar{\bx}^c_{ij},t^n)$ is computed using an approximate value of $\bar{\bx}^c_{ij}$ given by 
\begin{align*}
&\bar{\bx}^c_{ij}={\bx}^c_{ij} - \left( \left(\frac{f}{s}\right)({s}_{ij}^n,{c}_{ij}^n) \bv - \left(\frac{\bm{D}}{s}\right) (\bar{s}^n_{ij},{c}_{ij}^n)\nabla s\right) \Delta t/\phi,
\end{align*}
where, as before, $(\bar{\bx}^c_{ij},t^n)$ and $({\bx}^c_{ij},t^{n+1})$ lie on the same characteristic curve. Here the superscript `c' is used to denote the characteristic curves associated with the polymer transport equation \Eq{paper2-satpol}.
Thus we arrive at the following implicit-time finite difference formulation for \Eq{paper2-satpol}\\
\begin{align}\label{eq:paper2-pol}
\phi \frac{c_{ij}^{n+1}-\bar{c}_{ij}^n}{\Delta t} +(g)_{ij}^n c^{n+1}_{ij}&= (g_c)_{ij}^n.
\end{align}
Hence \Eq{eq:paper2-sat} and \Eq{eq:paper2-pol} form the finite difference approximation of the transport equations, \Eq{paper2-satnum} and \Eq{paper2-satpol} respectively. 

The pseudocode (see Algorithm~\ref{eor}) for the method is given below. Here, $s_0^{\sigma0}$ is the initial resident wetting phase saturation or the amount of water (wetting phase) present in the reservoir before the flood simulation starts and $Q = \abs{q}$ is the volumetric flux at the injection/production points.

\begin{algorithm}[ht]
\caption{Polymer flooding simulation}\label{eor}
\begin{algorithmic}[1]
\Procedure{}{}

\noindent \emph{Set up Cartesian grid and FE Mesh and a permeability field}
\vskip .4ex
\State $i,j \gets 1, \ldots , N; \; h \gets \dfrac{1}{N}$ \Comment{($N\times N$ \textit{is the grid size})}
\State  $\Sigma \gets $ \textit{Initial interface} \Comment{$\Sigma = \dl \Omega^+ \cup \dl \Omega^-$}
\State $\bK(\bx) \gets \textit{choose type of heterogeneity}$
\vskip 1ex
\noindent \emph{Set model parameters\;} 
\vskip .4ex
\State $\mu_o, \, \mu_w, \, s_{ro}, \, s_{ra}, \, Q, \, c_0, \, s_0^{\sigma0} \gets \textit{values from \Cref{table:paper2-input_par}}$
\vskip 1ex
\noindent \emph{Initialization\;}
\vskip .4ex
\State $(s,c) \gets \begin{cases}
    (1 - s_{ro}, c_0)  \hfill &  x \in \Omega^+\\
    (s_0^{\sigma0}, 0) \hfill &  x \in \Omega^-\\
    \end{cases}\;$
\State $t \gets 0$ \; \State $\Delta t \gets \textit{value} $\; \Comment{$\Delta t$ chosen for desired accuracy}
\vskip 1ex
\noindent \emph{Computation loop\;}
\vskip .4ex
\While{$\left(s(\bx_{N,N},t) \leq  1-s_0^{\sigma0} \text{ and } t < Tstop \right)$}
 \State Compute $\{\mu_a, \, \lb_a, \, \lb_o, \, \lb, \, p_c\}\; \text{using} \;(s^n, c^n,\bv^{n-1})$
 \State Solve the global pressure equation for $p^{n}, \bv^n \;$ 
 \State Recompute $\{\mu_a, \, \lb_a, \, \lb_o, \, \lb, \, p_c\}\; \text{using} \; (s^n, c^n, \bv^{n})$
 \State Solve the transport equations for $s^{n+1}$ and $c^{n+1}$	
 \State $t \gets t+ \Delta t$
 \State \textbf{close};
\EndWhile
\EndProcedure
\end{algorithmic}
\end{algorithm}

\section{Convergence study and error analysis}\label{sec:analysis}

Let $s_i^n = s(x_i,t^n)$ be the grid values of the true solution of the saturation equation \eqref{paper2-satnum} and  $w_i^n = w(x_i,t^n)$ be the grid values of the discrete approximate solution where $x_i = ih$ and $t^n = n\Delta t$. Similarly, let $p_i^n$ and $r_i^n$ be the grid values of the true and the discrete approximate solutions respectively of \cref{eq:paper2-ell-1}, $c^n_i$ and $m_i^n$ be the grid values of the true and the discrete solutions respectively of \cref{paper2-satpol}, and finally, let $v_i^n$ and $z_i^n$ be the grid values of the true and the discrete approximate solutions respectively of the total velocity given by $\bv = - \bK \lb \bgrad p$. Let the numerical approximation errors be defined as
\begin{align}\label{error-def}
\zeta^n_i = s_i^n - w_i^n, \quad \pi^n_i = p_i^n - r_i^n,\quad \&\quad \theta^n_i = c_i^n - m_i^n.
\end{align}
We define the following discrete norms for any $u \in W^{l,p}(\Omega)$, $v \in L^2(\Omega)$ and $w \in L^\infty(\Omega)$ where $\Omega = [0,1]^2$. 
\begin{align*}
&\nm{u}_{l,p} = \left(\sum_{k=0}^l \bigg( \sum_i h\abs*{\frac{d^ku_i}{dx^k}}^p \bigg) \right)^{1/p}, \quad  \abs{u}_{l,p} = \left(\sum_{i}h\abs*{\frac{d^lu_i}{d x^l}}^p \right)^{1/p}\\
&\nm{v} = \left(\sum_{i}h\abs*{v_i}^2 \right)^{1/2}, \qquad \nm{w}_\infty = \max_{i}{\abs{w_i}}
\end{align*}
In particular, $\nm{.}$ and $\langle .\rangle$ denote the discrete $L^2$ norm and the associated inner product respectively. 

For the analysis, we consider a reduced system of equations (see \cref{paper2-satnum,paper2-satpol}) in one spatial dimension as given by
\begin{subequations}
\begin{align}
\phi \pdt{s} + b \pdx{s} - \pdx{}\left(D \pdx{s}\right) &= F, \qquad s(x,0) = s_0(x) ; \quad x \in \Omega \setminus \dl \Omega \label{conv-sat1}\\
\phi \pdt{c} + a \pdx{c} + G c &= H , \qquad c(x,0) = c_0(x) ; \quad x \in \Omega \setminus \dl \Omega \label{conv-conc1}
\end{align}
\end{subequations}
where $b(s,c) = \p{f}{s} v$, $a(s,c) = (\frac{f}{s}v -\frac{D}{s}\pdx{s})$, $F(s,c)=g_s -  v \p{f}{c} \pdx{c}$, $G(s) = g$ and $H(s)=g_c$. Then the characteristic finite difference approximation of \Eq{conv-sat1} and \Eq{conv-conc1} are given by
\begin{align}
\phi_i \frac{w_i^n -\bar{w}_i^{n-1}}{\Delta t} - \delta_{x} (\bar{D} \delta_x w^n)_i &= F_i^n \qquad w_i^0 = s_0(x_i),\label{conv-num1} \\
\phi_i \frac{m_i^n -\bar{m}_i^{n-1}}{\Delta t} + G_i^n m_i^n &= H_i^n \qquad m_i^0 = c_0(x_i), \label{conv-num11}\\
\text{ where } \bar{w}_i^{n-1} = w(\tilde{x}^s_i,t^{n-1}), &\quad \tilde{x}^s_i = x_i - b(w_i^{n-1},m^{n-1}_i)\Delta t/\phi_i, \label{time-deriv1}\\ 
\text{ and } \bar{m}_i^{n-1} = m(\tilde{x}^c_i,t^{n-1}), &\quad \tilde{x}^c_i = x_i - a(w_i^{n},m^{n-1}_i)\Delta t/\phi_i. \label{time-deriv2}
\end{align}
From \Eq{conv-sat1} and \Eq{conv-num1} we have the following 
\begin{subequations}
\begin{align}
\psi_s \frac{\dl s}{\dl \tau} - \pdx{}\left(D \pdx{s}\right) &= F(s^n,c^n), \qquad s(x,0) = s_0(x) , \label{conv-sat2}\\
\phi_i \frac{w_i^n -\bar{w}_i^{n-1}}{\Delta t} - \delta_{x} (\bar{D} \delta_x w^n)_i &= F(w^n_i,m^n_i) \qquad w_i^0 = s_0(x_i). \label{conv-num2}
\end{align}
\end{subequations}
In the following analysis $M$, $\tilde{M}$, $\hat{M}$, $M_k (k \in \mathbb{Z}^+)$ and $C$ are generic constants, that depend on higher order norms of the true solutions $s,c$ and $p$, and are independent of the time step $\Delta t$ and space discretization $h$. We make assumptions on the smoothness of the coefficients of porosity and absolute permeability, that yield the bounds $0<\phi_*\leq \phi(x) \leq \phi^*$ and $\bK_* \leq \bK \leq \bK^* $. These have relatively minor mathematical significance in comparison to the other hypotheses introduced here. 
Note that the capillary-induced diffusion coefficients are non-negative, \ie $D,\bar{D} \geq 0$. This is ensured by the definition of capillary pressure, given in \Eq{paper2-cap}, which implies that $\frac{d \, p_c}{d \,s} \leq 0$. However, in the rest of the analysis (as in \cite{D1983}), the diffusion coefficients will be assumed to be bounded below by a positive constant \ie $D,\bar{D} \geq D_*$. Physically this is equivalent to treating problems in which $0<s_0^{\sigma0} \leq s \leq 1-s_{ro} $. The lower bound of $s$ also ensures that the functions $(f/s)$ and $(D/s)$ appearing in the component transport equation (\ref{paper2-satpol}) are well defined. Moreover, if and when required in the analysis, the auxiliary functions $f(s,c), \lb(s,c), D(s,c)$ have been assumed to have sufficient regularity in the discrete Sobolev norms. This follows from the definitions of relative permeabilities, capillary pressure and the aqueous phase viscosity given in \cref{paper2-eq:relperm1,paper2-eq:relperm2,paper2-eq:cap,paper2-eq:vis} which guarantee that the auxiliary functions and their higher order derivatives are continuous and bounded. Finally, the assumption of smoothly distributed external flow rates $q_a$ and $q_o$ implies that the source terms are sufficiently smooth and bounded as given by $\abs{F}, \abs{H} \leq M$. 

The primary focus in the analysis is on the convergence behavior of the discrete approximation to the transport equations. Hence, the approximation error for the pressure discretization using the non-traditional discontinuous finite element method will be estimated by \cite{HWW2010},
\begin{align}\label{errorD}
 \nm*{\pdx{\pi^n}}_{\infty} = O(h). 
\end{align} 
The regularity assumptions on the source terms ensure that the conditions of Theorem 2.1 in \cite{HWW2010} are met. This, along with the normalization constraint on pressure, guarantees the existence and uniqueness of a weak solution~\cite{HWW2010,HL2005} to the elliptic pressure \cref{eq:paper2-ell-1} and thus validates the convergence estimate in \cref{errorD}. The numerical scheme converges for any other finite element formulation that preserves or improves upon the above error estimate.

The salient steps in the analysis are outlined as follows. The numerical method is decomposed into a few fundamental approximation steps and the errors introduced at each step are estimated, namely, the discretization procedure, the linear interpolation steps and the linearization of coefficients. In the first part of the analysis, we study the aqueous phase transport equation \eqref{conv-sat1}, followed by the component (polymer) transport equation \eqref{conv-conc1}. Finally the results are combined together to obtain an unified error estimate and to prove the convergence of the numerical method for the transport system. 

To begin, the discretization error for the characteristic derivative and the second order spatial derivative are estimated using Taylor series expansions. Then, the error associated with the linear interpolation of nodal values to compute solutions at points where the characteristic curves intersect the computational grid are estimated using a combination of the Peano kernel theorem and Taylor series expansions. This is done in several steps as an exact characteristic curve and its numerical approximation may intersect the computational grid at different points leading to further interpolation errors. The transport \cref{conv-sat2,conv-num2} are combined with these estimates to obtain a representation in terms of the inner products of the error variables $\zeta_i^n$, $\theta_i^n$ and $\pi_i^n$ given by \cref{error-def}. This is followed by estimation of the time discretization error, and the error due to linearization of the auxiliary functions, their derivatives and the capillary dissipation term using Cauchy-Schwarz inequality, a modified Young's inequality and some relevant results from literature. Finally an application of the discrete Gr\"{o}nwall's inequality leads us to the partial convergence result \eqref{conv-sat1} for the water transport equation. Following a similar procedure with the polymer transport equation and combining the estimates together, we obtain the final convergence result \eqref{final-convergence} and the error estimates \eqref{final-estimate} for the transport variables. 

For the rest of the analysis of the water saturation equation \eqref{conv-num2}, with slight abuse of notation, we  will write $\tilde{x}_i^n$ and $\bar{x}_i^n$ to mean $\tilde{x}_i^{s,n}$ and $\bar{x}_i^{s,n}$ respectively. In the next two lemmas we estimate the errors introduced by approximating the derivative in the characteristic direction and the second order derivative term in \cref{conv-sat2} with their finite difference discretizations given in \cref{conv-num2}.
\begin{lemma}\label{lem:char-error}
The error in approximating the characteristic derivative in \cref{conv-sat2} is given by
\begin{align*} 
\psi_{s,i} \left(\p{s}{\tau}\right)_i^n - \phi_i \frac{s_i^n -\bar{s}_i^{n-1}}{\Delta t} = O\left(\left|\frac{\dl^2 s^{*}}{\dl \tau^2}\right| \Delta \tau\right),
\end{align*}
\end{lemma}
where $\bar{s}_i^{n-1} = s(\bar{x}_i,t^{n-1})$ with $\bar{x}_i = x_i - b(s^n_i,c^n_i) \Delta t/\phi_i$.
\begin{proof} 
Let $p_1 = (x,t^n)$ be a point on the grid (see \Fig{fig:paper2-char-method}) and the characteristic that passes through this point intersects the previous time level at $p_2 = (\bar{x},t^{n-1})$ where $\bar{x} = x - b(s,c)/\phi(x) \Delta t$ and let $\Delta \tau = \left[(x-\bar{x})^2 + (t^n -t^{n-1})^2\right]^{1/2}$. Hence $\Delta \tau = \frac{\psi_s}{\phi} \Delta t$.  Using the Taylor series expansion along the characteristic direction, we write
\begin{align*}
 &s(p_1-\Delta \tau) = s(p_1) - \Delta \tau \p{s}{\tau} + \frac{\Delta \tau^2}{2} \frac{\dl^2 s^*}{\dl \tau^2}, \ 
\shortintertext{where $\frac{\dl^2 s^*}{\dl \tau^2}$ is some evaluation of the second derivative along the characteristic segment between $p_2$ and $p_1$. In the convection dominated case, this second derivative is relatively much smaller than $\frac{\dl ^2 s}{\dl x^2}$ or $\frac{\dl ^2 s}{\dl t^2}$ \cite{DR1982}. This is rewritten as}
 &\Delta \tau \p{s^n}{\tau} = s^n -\bar{s}^{n-1} + \frac{\Delta \tau^2}{2} \frac{\dl^2 s^*}{\dl \tau^2} .\
\shortintertext{Using $\displaystyle{\frac{\psi_s}{\Delta \tau} = \frac{\phi}{\Delta t}}$ we obtain}
 &\psi_s \p{s^n}{\tau} - \phi \frac{s^n -\bar{s}^{n-1}}{\Delta t} = \frac{\phi}{\Delta t} \frac{\Delta \tau^2}{2} \frac{1}{2} \frac{\dl^2 s^*}{\dl \tau^2} = \frac{\psi_s}{\Delta \tau} \frac{\Delta \tau^2}{2} \frac{\dl^2 s^*}{\dl \tau^2}  = \frac{\psi_s}{2}\Delta \tau \frac{\dl^2 s^*}{\dl \tau^2} .\
\shortintertext{This leads to the final result:}
 &\psi_{s,i} \left(\p{s}{\tau}\right)^n_i - \phi_i \frac{s_i^n -\bar{s}_i^{n-1}}{\Delta t} = O\left(\left|\frac{\dl^2 s^{*}}{\dl \tau^2}\right| \Delta \tau\right).  \qed
\end{align*}
\end{proof}
Using Lemma \ref{lem:char-error} in \Eq{conv-sat2}, we estimate the  error introduced by numerical discretization of the characteristic derivative as 
\begin{align}\label{errorA}
\phi_i \frac{s_i^n -\bar{s}_i^{n-1}}{\Delta t} - \pdx{}\left(D \pdx{s^n}\right)_i &= F(s^n_i,c^n_i) + O\left(\left|\frac{\dl^2 s^{*}}{\dl \tau^2}\right| \Delta \tau\right).
\end{align}

Now we estimate the approximation error for the second order derivative term in the left hand side of \Eq{conv-sat2}. By definition,
\begin{align}
&\delta_{x} (\bar{D} \delta_x w^n)_i = \frac{1}{h} \left( \bar{D}_{i+1/2}(\delta_x w^n)_{i+1/2} - \bar{D}_{i-1/2}(\delta_x w^n)_{i-1/2} \right) \nonumber \\
&\qquad \qquad \quad = \frac{1}{h^2} \left( \bar{D}_{i+1/2}(w_{i+1}^n- w_i^n) - \bar{D}_{i-1/2}(w_i^n - w_{i-1}^n) \right), \label{errorA-1} \\
&\text{ where} \nonumber \\ 
&\bar{D}_{i+1/2} = \frac{1}{2} \left[D(x_i,\bar{w}^{n-1}_i) + D(x_{i+1},\bar{w}^{n-1}_{i+1}) \right] \quad \& \quad 
\bar{D}_{i-1/2} = \frac{1}{2} \left[D(x_i,\bar{w}^{n-1}_i) + D(x_{i-1},\bar{w}^{n-1}_{i-1}) \right]. \nonumber 
\end{align}
The numerical approximation of the second order derivative in \Eq{conv-sat1} is given by
\begin{align*}
&\delta_x(D \delta_x s^n)_i = \frac{1}{h^2} \left[D_{i+1/2} (s^n_{i+1} -s^n_i) - D_{i-1/2} (s^n_{i} -s^n_{i-1}) \right] , \\
&\text{ where} \\
&D_{i+1/2} = \frac{1}{2} \left[D(x_i,s^n_i) + D(x_{i+1},s^n_{i+1}) \right] \quad \& \quad
D_{i-1/2} = \frac{1}{2} \left[D(x_i,s^n_i) + D(x_{i-1},s^n_{i-1}) \right].
\end{align*}

\begin{lemma}\label{lem:secdev-error}
The finite difference approximation error of the second derivative term in \cref{conv-sat2} is given by
\begin{align*}
\frac{d}{dx} \left( D \frac{d}{dx} s^n \right)_i - \delta_{x} (D \delta_x s^n)_i = O(h \nm*{s^n }_{3,\infty}).
\end{align*}
\end{lemma}

\begin{proof} From the Taylor series expansion, we know that
\begin{align*}
&\left(\frac{du}{dx}\right) - \frac{u(x+h/2)-u(x-h/2)}{h} = O \left(h^2 \nm{u}_{3,\infty} \right). \
\shortintertext{This can be rewritten as}
&\left(\frac{du}{dx}\right)_i - \delta_x(u_i) = O(h^2 \nm{u}_{3,\infty}). \
\shortintertext{Using this estimate for the second derivative term, we obtain} 
& \frac{d}{dx} \left(D \frac{d}{dx} s^n\right)_i - \delta_{x} (D \delta_x s^n)_i = O\left(h^2 \nm*{D \frac{d}{dx} s^n}_{3,\infty}\right). \
\shortintertext{Using the defintion \eqref{errorA-1}, we obtain}
& \frac{d}{dx} \left(D \frac{d}{dx} s^n\right)_i - \frac{1}{h} \left( \left(D\frac{d s^n}{dx}\right)_{i+1/2} -\left(D\frac{d s^n}{dx}\right)_{i-1/2} \right) = O\left(h^2 \nm*{D \frac{d}{dx} s^n}_{3,\infty}\right). \
\shortintertext{Using $\left(D\dfrac{d s^n}{dx}\right)_{i+1/2} = D_{i+1/2} \left(\dfrac{ds^n}{dx}\right)_{i+1/2} = D_{i+1/2} \left(\dfrac{s^n_{i+1} -s_i^n}{h}\right) +  O \left(h^2 \nm{s^n}_{3,\infty} \right) $, we continue as}
&\frac{d}{dx} \left(D \frac{d}{dx} s^n \right)_i  - \frac{1}{h} \left[ D_{i+1/2} \frac{s^n_{i+1} -s_i^n}{h}  O(h^2 \nm{s^n}_{3,\infty})\right]\\
& \qquad \qquad +\frac{1}{h}\left[ D_{i-1/2} \frac{s^n_{i} -s_{i-1}^n}{h} + O(h^2 \nm{s^n}_{3,\infty}) \right] = O\left(h^2 \nm*{D}_{\infty} \, \nm*{\frac{d s^n}{dx}}_{3,\infty}\right).
\end{align*}
This leads to the final estimate
\begin{align*}
\frac{d}{dx} \left(D \frac{d}{dx} s^n\right)_i  &- \frac{1}{h} \left[ D_{i+1/2} \frac{s^n_{i+1} -s_i^n}{h} - D_{i-1/2} \frac{s^n_{i} - s_{i-1}^n}{h} \right] \\
  &= O(h \nm{s^n}_{3,\infty}) + O(h^2 \nm{s^n}_{4,\infty}) = O(h \nm{s^n}_{3,\infty}). \qed
\end{align*}
\end{proof}

Lemmas \ref{lem:char-error} and \ref{lem:secdev-error} will be applied to the difference of \cref{conv-sat2,conv-num2} as the initial step in estimating the numerical approximation error for the saturation equation. Accordingly, using the result of Lemma \ref{lem:secdev-error}, we rewrite \Eq{errorA} as
\begin{align}\label{errorB}
\phi_i \frac{s_i^n -\bar{s}_i^{n-1}}{\Delta t} - \delta_x(D\delta_x s^n)_i = F(s_i^n,c^n_i) + O\left(\left|\frac{\dl^2 s^{*}}{\dl \tau^2}\right| \Delta \tau\right) + O(h \nm*{s^n }_{3,\infty}).
\end{align}
Subtracting \Eq{conv-num2} from \Eq{errorB}, we obtain
\begin{align}\label{errorB1}
\begin{split}
\phi_i \frac{s_i^n -\bar{s}_i^{n-1}}{\Delta t} - \phi_i \frac{w_i^n -\bar{w}_i^{n-1}}{\Delta t} & - \delta_x(D\delta_x s^n)_i + \delta_x(\bar{D}\delta_x w^n)_i \\
 &= F(s_i^n,c^n_i) - F(w_i^n,m_i^n)  + O\left(\left|\frac{\dl^2 s^{*}}{\dl \tau^2}\right| \Delta \tau, h \nm*{s^n }_{3,\infty} \right). 
\end{split}
\end{align}
Using the definition of saturation error $\zeta^n_i$ from \cref{error-def} in \cref{errorB1} and rearranging terms, we rewrite the above as
\begin{align}
\phi_i \frac{\zeta_i^n -(\bar{s}_i^{n-1} - \bar{w}_i^{n-1})}{\Delta t} &- \delta_x(\bar{D}\delta_x \zeta^n)_i = F(s_i^n,c^n_i) - F(w_i^n,m_i^n) \nonumber \\
&+ O\left(\left|\frac{\dl^2 s^{*}}{\dl \tau^2}\right| \Delta \tau, h \nm*{s^n }_{3,\infty} \right) -\delta_x((\bar{D}-D)\delta_x s^n)_i \, . \label{errorC}
\end{align}

The estimation of the various terms in \cref{errorC} will be carried out sequentially, starting with the estimates of the temporal discretization errors, the bilinear interpolation errors and the numerical approximation of the source terms. 
Consider the first term on the left hand side of \Eq{errorC}. Let $\zeta^n = \Ia \zeta^n_i$  be the piecewise linear interpolant of $\zeta^n_i$ such that $\bar{\zeta}_i^{n-1} = \Ia\zeta^{n-1}(\tilde{x}_i) = \Ia s^{n-1}(\tilde{x}_i) - w^{n-1}(\tilde{x}_i) = \Ia s^{n-1}(\tilde{x}_i) - \bar{w}_i^{n-1}$. Then,
\begin{align}
\zeta_i^n -(\bar{s}_i^{n-1} - \bar{w}_i^{n-1}) &= (\zeta^n_i - \bar{\zeta}^{n-1}_i) + \Ia s^{n-1}(\tilde{x}_i) - \bar{w}_i^{n-1} - \bar{s}_i^{n-1}+ \bar{w}_i^{n-1} \nonumber \\
&= (\zeta^n_i - \bar{\zeta}^{n-1}_i) - \underbrace{(s^{n-1}(\bar{x}_i) - s^{n-1}(\tilde{x}_i))}_{\text{A}} - \underbrace{((1-\Ia)s^{n-1}(\tilde{x}_i))}_{\text{B}}. \label{errorC1}
\end{align}
Below, we find estimates for the last two terms, A and B, of the right hand side of \Eq{errorC1}, followed by the estimate of the source term
$(F(s_i^n,c_i^n) - F(w_i^n,m_i^n))$  on the right hand side of \Eq{errorC}. Once we have these estimates, we can substitute \Eq{errorC1} in \Eq{errorC}, take inner products with $\zeta_i^n$ and use the estimates to rewrite the equation.
\smallskip

\noindent{\underline{\bf A. Estimate of the term A  on the right hand side of \Eq{errorC1}}}: This is
carried out in several steps below.
\begin{align}
&|s^{n-1}(\bar{x}_1) - s^{n-1}(\tilde{x}_i)| \nonumber \\
&\leq \hat{M} \nm{s^{n-1}}_{1,\infty} \, |\bar{x}_i - \tilde{x}_i| \qquad (\hat{M} \text{ is a constant } ) \nonumber\\
& = \hat{M} \nm{s^{n-1}}_{1,\infty}\, \abs*{ \frac{\dl f}{\dl s}(w_i^{n-1},m^{n-1}_i) z_i^{n-1} - \frac{\dl f}{\dl s}(s_i^n,c^n_i) v_i^n } \frac{\Delta t}{\phi_i} \nonumber\\
& \leq \hat{M} \nm{s^{n-1}}_{1,\infty}\,\left( \abs*{ \frac{\dl f}{\dl s}(w_i^{n-1},m^{n-1}_i)} \underbrace{\abs{ z_i^{n-1}-v_i^n }}_{\text{A-1}} \right. \nonumber \\
& \qquad \qquad \qquad \left. + \abs{ v_i^n} \underbrace{\abs*{ \p{f}{s}(w_i^{n-1},m^{n-1}_i) - \p{f}{s}(s_i^n,c^n_i) }}_{\text{A-2}} \right) \frac{\Delta t}{\phi_i}. \label{errorC2}
\end{align}
Next we estimate the terms A-1 and A-2 of  the right hand side of \eqref{errorC2}.

\noindent{\underline{\bf A-1. Estimate of the term A-1  on the right hand side of \Eq{errorC2}}}: 

\noindent We rewrite the term A-1 as
\begin{equation}\label{errorD1}
 \abs{z_i^{n-1} - v_i^{n}} \leq \underbrace{\abs{z_i^{n-1} - v_i^{n-1}}}_{\text{A-1-1}} +\underbrace{\abs{v_i^{n} - v_i^{n-1}}}_{\text{A-1-2}}
\end{equation}
Recall that $ z_i^{n} = - K \lb(w_i^{n},m_i^{n})\pdx{r_i^{n}}$ and $ v_i^n = - K \lb(s_i^n,c_i^n)\pdx{p_i^n}$. Then 
the first term  A-1-1 on the right hand side of the above inequality  \eqref{errorD1} is written as
\begin{align}
 &\abs{z_i^{n-1} - v_i^{n-1}} \nonumber\\
 =&\abs*{K\lb(w_i^{n-1},m_i^{n-1})\pdx{}(p_i^{n-1}-r_i^{n-1})
  \quad +K\left(\lb(s_i^{n-1},c_i^{n-1})-\lb(w_i^{n-1},m_i^{n-1})\right)\pdx{p_i^{n-1}}} \nonumber\\
 \leq & \nm{K}_\infty \, \nm{\lb}_\infty\, \nm*{\pdx{}(\pi_i^{n-1})}_\infty 
 \quad + \nm{K}_\infty \,\abs*{\lb(s_i^{n-1},c_i^{n-1})-\lb(w_i^{n-1},m_i^{n-1})}\, \nm*{\pdx{}(p_i^{n-1})}_\infty. \label{errorE}
\end{align}
Using Taylor series we write,  
\begin{align}
&\abs*{\lb(s_i^{n-1},c_i^{n-1})-\lb(w_i^{n-1},m_i^{n-1})} \\
& \quad \leq \abs*{s_i^{n-1} -w_i^{n-1}}\, \nm*{\p{\lb}{s}}_\infty
 + \abs*{c_i^{n-1}-m_i^{n-1}} \, \nm*{\p{\lb}{c}}_\infty \nonumber \\
& \quad \leq \bar{M} (\abs*{\zeta_i^{n-1}}+ \abs*{\theta_i^{n-1}}) . \label{errorE1}
\end{align}
Using \Eq{errorE1} and the estimate of the pressure error gradient given by \Eq{errorD} in \Eq{errorE}, we obtain following estimate for the first term A-1-1 of the right hand side of  \eqref{errorD1}.
\begin{align}\label{errorE2}
\abs{z_i^{n-1} - v_i^{n-1}} \leq  M(h+\abs{\zeta_i^{n-1}}+ \abs{\theta_i^{n-1}} ).
\end{align}
To estimate the term A-1-2 of the inequality \eqref{errorD1} we observe
\begin{align}\label{errorE3}
\abs{v_i^n - v_i^{n-1}} \leq \Delta t \nm*{\pdt{v}}_\infty .
\end{align} 
Using \Eq{errorE2} and \Eq{errorE3} in \eqref{errorD1}, we obtain the following estimate for A-1 (see  \Eq{errorC2}).
\begin{align}\label{errorF}
\abs{z_i^{n-1} - v_i^{n}} \leq M(h+\Delta t + \abs{\zeta_i^{n-1}}+ \abs{\theta_i^{n-1}}).
\end{align}
This concludes the estimate for the term A-1 in \Eq{errorC2}. 

\noindent{\underline{\bf A-2. Estimate of the term A-2  on the right hand side of \Eq{errorC2}}}:  
\begin{align}
&\abs*{ \p{f}{s}(w_i^{n-1},m_i^{n-1}) - \p{f}{s}(s_i^n,c_i^n) } \nonumber \\
& \leq \abs*{(w_i^{n-1}-s_i^{n-1})}\nm*{\pderiv[2]{f}{s}}_\infty + \abs*{m_i^{n-1} - c_i^{n-1}}\nm*{\pderiv[1]{f}{c,s}}_\infty \nonumber \\
& \qquad + \abs*{s_i^{n-1} - s_i^n}\nm*{\pderiv[2]{f}{s}}_\infty \nm*{\pdt{s}}_\infty  + \abs*{c_i^{n-1} - c_i^n}\nm*{\pderiv[1]{f}{c,s}}_\infty \nm*{\pdt{c}}_\infty\nonumber\\
&  \leq M (\abs{\zeta_i^{n-1}} + \abs{\theta_i^{n-1}}+\Delta t). \label{errorG}
\end{align}
Using the estimates for A-1 and A-2, as given by \Eq{errorF} and \Eq{errorG} respectively, in \Eq{errorC2} we finally obtain the estimate for the term A of \Eq{errorC1} as
\begin{align}\label{errorH}
\abs{s^{n-1}(\bar{x_i}) - s^{n-1}(\tilde{x_i})} \leq M \Delta t ( \abs{\zeta_i^{n-1}} + \abs{\theta_i^{n-1}} +h + \Delta t).
\end{align}

\noindent{\underline{\bf B. Estimate of the term B  on the right hand side of \Eq{errorC1}}}:  
\smallskip

\noindent Using the Peano kernel Theorem in the spirit of the analysis in Douglas and Russell~\cite{DR1982}, we obtain the following,
\begin{align}\label{pkestimate}
(1-\Ia)s^{n-1}(\tilde{x}_i) = O \left(h^2\nm*{s^{n-1} }_{2,\infty} \right).
\end{align}

\noindent{\underline{\bf C. Estimate of the source term $(F(s_i^n,c_i^n) - F(w_i^n,m_i^n) )$ in \Eq{errorC}}}: 
\begin{align}
F(s_i^n,c_i^n) - F(w_i^n,m_i^n) &\leq \abs{s_i^n - w_i^n} \nm*{\p{F}{s}}_\infty + \abs{c_i^n - m_i^n}\nm*{\p{F}{c}}_\infty \nonumber \\
&\leq M( \abs{\zeta_i^n} + \abs{\theta_i^n}). \label{errorI}
\end{align}

\noindent Equation \eqref{errorC1} is substituted into \Eq{errorC} and the resulting equation is tested against $\zeta_i^n$. Using the estimates \eqref{errorH}, \eqref{pkestimate} and \eqref{errorI} to replace some of the inner products, we rewrite \Eq{errorC} as
\begin{align}\label{errorJ}
\underbrace{\langle \phi_i \frac{\zeta_i^n - \bar{\zeta}_i^{n-1}}{\Delta t} , \zeta_i^n \rangle}_{\text{D-1}} &- \underbrace{\langle \delta_x (\bar{D} \delta_x \zeta^n)_i, \zeta_i^n \rangle}_{\text{D-2}} \leq \langle M \left( h + \Delta t + h^2/ \Delta t + \abs{\zeta_i^{n-1}} + \abs{\theta_i^{n-1}} + \abs{\zeta_i^{n}} + \abs{\theta_i^{n}} \right) , \zeta_i^n \rangle \nonumber \\
 & \qquad \qquad \qquad + \underbrace{\langle \epsilon_i^n , \zeta_i^n \rangle}_{\text{D-3}} - \underbrace{\langle \delta_x ((\bar{D}-D) \delta_x s^n)_i, \zeta_i^n \rangle }_{\text{D-4}} 
\end{align}
where $ \epsilon^n_i = O(\nm{\pderiv[2]{s}{\tau}}_\infty \Delta \tau, \nm{s^n}_{3,\infty} h)$.
Our objective here, would be to estimate the inner products in D-1, D-2, D-3 and D-4 in terms of the norms of the errors $\zeta$ and $\theta$.
\smallskip 

\noindent{\underline{\bf D-1. Estimate of the term D-1}}: The inner product is rewritten as 
\begin{align*}
\langle \phi_i \frac{\zeta_i^n - \bar{\zeta}_i^{n-1}}{\Delta t} , \zeta_i^n \rangle =  \underbrace{\langle \phi_i \frac{\zeta_i^n - {\zeta}_i^{n-1}}{\Delta t} , \zeta_i^n\rangle}_{\text{D-1-1}} - \langle \phi_i \frac{\bar{\zeta}_i^{n-1} - {\zeta}_i^{n-1}}{\Delta t} , \zeta_i^n \rangle
\end{align*}
Using the inequality $ \abs{a-b}\abs{a} \geq \dfrac{\abs{a}^2 - \abs{b}^2}{2}$ we estimate the term D-1-1 as
\begin{align*}
\langle \phi_i \frac{\zeta_i^n - {\zeta}_i^{n-1}}{\Delta t} , \zeta_i^n \rangle \geq \frac{M}{\Delta t} (\nm{\zeta^n}^2 - \nm{\zeta^{n-1}}^2)
\end{align*}

\noindent{\underline{\bf D-2. Estimate of the term D-2 in \Eq{errorJ}}}: 

Using summation by parts, we write
\begin{align*}
\langle \delta_x (\bar{D} \delta_x \zeta^n)_i, \zeta_i^n \rangle = - \langle (\bar{D}\delta_x \zeta^n)_i, (\delta_x \zeta^n)_i \rangle
\end{align*}
and similarly for D-4, we have
\begin{align*}
\langle \delta_x ((\bar{D}-D) \delta_x s^n)_i, \zeta_i^n \rangle = - \langle ((\bar{D}-D)\delta_x s^n)_i, (\delta_x \zeta^n)_i \rangle.
\end{align*}

\noindent{\underline{\bf D-3. Estimate of the term D-3 in \Eq{errorJ}}}:
Testing the term D-3 against $\zeta^n_i$, we get
\begin{align*}
\langle \epsilon_i^n, \zeta_i^n \rangle \leq M (h+ \Delta t) \sum_i h \abs{\zeta^n_i} \leq M (h^2 + \Delta t^2 + \nm{\zeta^n}^2)
\end{align*}
Substituting the estimates for the term D-1-1 and replacing the terms D-2, D-3 and D-4 as shown before, we rewrite \Eq{errorJ} as
\begin{align}\label{errorK}
\frac{M}{\Delta t} (\nm{\zeta^n}^2 - \nm{\zeta^{n-1}}^2) +\underbrace{\langle (\bar{D}\delta_x \zeta^n)_i, (\delta_x \zeta^n)_i \rangle}_{\text{E-1}} \leq \hat{M} \bigg(\nm{\zeta^n}^2 + \underbrace{\langle \theta^n_i, \zeta_i^n \rangle}_{\text{E-2}} + h^2 + \Delta t^2 + \frac{h^4}{\Delta t^2} \bigg. \nonumber \\
\bigg. + \nm{\zeta^{n-1}}^2 + \nm{\theta^{n-1}}^2\bigg) + \underbrace{\langle \left((\bar{D}-D)\delta_x s^n \right)_i, \left(\delta_x \zeta^n \right)_i \rangle}_{\text{E-3}} + \underbrace{\langle \phi_i \frac{\bar{\zeta}_i^{n-1} - {\zeta}_i^{n-1}}{\Delta t} , \zeta_i^n \rangle}_{\text{E4}}
\end{align}
The next step would be to estimate the inner product terms E-1, E-2, E-3 and E-4 in \Eq{errorK}. The regularity assumptions and lower bound on the capillary induced diffusion coefficients $D$ and $\bar{D}$ will be used in the estimation of E-1 and E-3.
\begin{align*}
&\underline{\textbf{E-1}}: \quad \langle (\bar{D}\delta_x \zeta^n)_i, (\delta_x \zeta^n)_i \rangle \geq  D_* \abs{\zeta^n}^2_{1,2}, \\
&\underline{\textbf{E-2}}: \quad \langle \theta^n_i, \zeta_i^n \rangle \leq M (\nm{\theta^n}^2+ \nm{\zeta^n}^2) \mbox{ (using Cauchy-Schwarz) }, \\
&\underline{\textbf{E-3}}: \quad \langle ((\bar{D}-D)\delta_x s^n)_i, (\delta_x \zeta^n)_i \rangle = \sum_i h \abs{(\bar{D}-D) \delta_x s^n}_i \abs{\delta_x \zeta^n}_i \\
& \qquad \qquad \qquad \leq \nm*{\pdx{s_i^n}}_\infty \sum_i h \abs*{D(\bar{w}_i^{n-1},m_i^{n-1}) - D(s^n_i,c_i^n)}\, \abs*{\pdx{\zeta_i^n}}.
\end{align*}
Using Taylor series we write
\begin{align*}
\abs*{D(\bar{w}_i^{n-1},m_i^{n-1}) - D(s^n_i,c_i^n)} &\leq \abs{s_i^n - \bar{w}_i^{n-1}} \nm*{\p{D}{s}}_\infty + \abs{c_i^n - m_i^{n-1}} \nm*{\p{D}{c}}_\infty \\
&\leq M(\Delta t + h \Delta t + h^2 + \abs*{\bar{\zeta}_i^{n-1}} + \abs{\theta_i^{n-1}}),
\end{align*} 
where we use the estimate $ \abs{c_i^n - m_i^{n-1}} \leq M(\abs{\theta_i^{n-1}} + \Delta t) $. Also,
\begin{align*} 
\abs{s_i^n - \bar{w}_i^{n-1}} &\leq  \abs{s_i^n - s_i^{n-1}} +  \abs{s_i^{n-1} - \bar{w}_i^{n-1}} \leq \abs{s_i^n - s_i^{n-1}} + \abs*{\Ia s^{n-1}(\tilde{x_i}) - s_i^{n-1} - \bar{\zeta}_i^{n-1}}\\ 
&\leq M \left(\Delta t + \abs*{\tilde{x_i} - x_i} \nm{s^{n-1}}_{1,\infty} + \abs*{(1-\Ia) s^{n-1}(\tilde{x_i})} + \abs*{\bar{\zeta}_i^{n-1}}\right)\\
&\leq M \left( \Delta t  + \nm*{\p{f}{s}}_\infty \nm{z^{n-1}}_\infty \, \frac{\Delta t}{\phi_*} + C h^2 + \abs*{\bar{\zeta}_i^{n-1}} \right) \\ 
&\leq M \left(\Delta t + h \Delta t + h^2 + \abs*{\bar{\zeta}_i^{n-1}} \right) 
\end{align*} 
Above, we have used \Eq{pkestimate} and that $\nm{z^{n-1}}_\infty$ is bounded which will be proven below. Hence we have an estimate for E-3 as
\begin{align}
\langle \left((\bar{D}-D)\delta_x s^n \right)_i, \left(\delta_x \zeta^n \right)_i \rangle &\leq M_1 \nm{s^n}_{1,\infty}\sum_i h \abs{\bar{\zeta}_i^{n-1}} \, \abs*{\pdx{\zeta_i^n}} + M_2 \nm{s^n}_{1,\infty}\sum_i h \abs{\theta_i^{n-1}} \, \abs*{\pdx{\zeta_i^n}} \nonumber \\
&   \qquad   + M_3 \nm{s^n}_{1,\infty}\sum_i h (\Delta t + h \Delta t + h^2)\abs*{\pdx{\zeta_i^n}} \nonumber \\
&    \leq M ( \nm{\bar{\zeta}^{n-1}}^2 + \nm{\theta^{n-1}}^2 + \abs{\zeta^n}^2_{1,2} + \Delta t^2 + h^2 \Delta t^2 + h^4). \label{errorK1}
\end{align}

\noindent\underline{\bf{ E-4. Estimate of the term E-4 in \Eq{errorK}}}: Using the fundamental theorem of calculus,
\begin{align*}
\bar{\zeta}_i^{n-1} - \zeta_i^{n-1} &= \int_{x_i}^{\tilde{x_i}} \pdx{\zeta^{n-1}} \, \frac{\tilde{x_i}- x_i}{\abs{\tilde{x_i}- x_i}} \, d\sigma \\
\text{Hence } \quad \abs{\bar{\zeta}_i^{n-1} - \zeta_i^{n-1}}&\leq \int_{x_i}^{\tilde{x_i}} \abs*{ \pdx{\zeta^{n-1}}}\, d\sigma \, \leq \left(\int_{x_i}^{\tilde{x_i}} d\sigma\right)^{1/2}  \left(\int_{x_i}^{\tilde{x_i}}\abs*{ \pdx{\zeta^{n-1}}}^2 \,d\sigma \right)^{1/2}.
\end{align*}
Therefore, 
\begin{align}
 \langle \phi_i \frac{\bar{\zeta}_i^{n-1} - {\zeta}_i^{n-1}}{\Delta t} , \zeta_i^n \rangle &\leq \frac{\phi^*}{\Delta t} \left(\sum_i h \abs{\zeta_i^n}^2 \right)^{1/2} \left( \sum_i h \abs{\bar{\zeta}_i^{n-1}- \zeta^{n-1}_i}^2\right)^{1/2} \nonumber \\
& \leq M  \abs{\zeta^{n-1}}_{1,2} \, \abs{\zeta^{n}}_\infty \nm{z^{n-1}}_\infty \nonumber \\
& \leq M \abs{\zeta^{n-1}}_{1,2} \, \abs{\zeta^{n}}_{1,2} \, (1+h) \left(\log{1/h}\right)^{1/2} \quad [\mbox{ Using a result from \cite{B1966}}] \nonumber \\
& \leq M \left(\log{1/h}\right)^{1/2} (1+h) \left(\abs{\zeta^{n-1}}^2_{1,2} + \abs{\zeta^{n}}^2_{1,2} \right). \label{errorK2}
\end{align}
Above, we have again used that $\nm{z^{n-1}}_\infty$ is bounded. Finally, all the ingredients that enable us to represent the water saturation equation in terms of the error norms have been obtained. 
Using the estimates for E-1, E-2, E-3, E-4 in \Eq{errorK} we get,
\begin{align}\label{errorL}
M &(\nm{\zeta^n}^2 - \nm{\zeta^{n-1}}^2) + D_* \Delta t \abs{\zeta^n}^2_{1,2}  \nonumber \\
&\leq M \Delta t (h^2 + \Delta t^2 + \frac{h^4}{\Delta t^2}+ h^2\Delta t^2 + h^4) + M \Delta t \left( \nm{\zeta^n}^2 + \nm{\theta^n}^2 + \nm{\zeta^{n-1}}^2 + \nm{\theta^{n-1}}^2 \right) \\
& \quad+ M \Delta t ( 1+ (1+h)\left(\log{1/h}\right)^{1/2})\left(\abs{\zeta^{n-1}}^2_{1,2} + \abs{\zeta^{n}}^2_{1,2} \right). \nonumber
\end{align}
Summing over $ 1\leq n \leq L $ (with $L\Delta t =T$),
\begin{align*}
&M (\nm{\zeta^L}^2 - \nm{\zeta^0}^2) + D_* \Delta t \sum_{n=1}^L\abs{\zeta^n}^2_{1,2}  \\
&\leq M T (h^2 + \Delta t^2 + \frac{h^4}{\Delta t^2} + h^2\Delta t^2 + h^4) + M \Delta t \sum_{n = 1}^{L} \left( \nm{\zeta^n}^2 + \nm{\theta^n}^2 + \nm{\zeta^{n-1}}^2 + \nm{\theta^{n-1}}^2 \right) \\
&+ M \Delta t \left(1+ (1+h) \left(\log{1/h}\right)^{1/2}\right)\sum_{n = 1}^{L} \abs{\zeta^n}^2_{1,2}.
\end{align*}
Using discrete Gr\"{o}nwall's inequality and noting that $\zeta^0_i =0 $ and $\theta^0_i =0 $ this can be rewritten as
\begin{align}\label{sat-conv1}
M \nm{\zeta^L}^2 + (D_* \Delta t - \rho_1) \sum_{n=1}^L\abs{\zeta^n}^2_{1,2} 
\leq M \Delta t \sum_{n = 1}^{L} \left( \nm{\theta^n}^2\right) + M \, \max(h^2+ \Delta t ^2, h^4/\Delta t^2)
\end{align}
where $\rho_1 = M \Delta t  (1+ (1+h) \left(\log{1/h}\right)^{1/2}) \rightarrow 0$ faster than $D_* \Delta t$ as $(h,\Delta t) \rightarrow 0$.
This concludes the analysis of the water transport equation (\Eq{conv-sat1}).

Next we consider the polymer transport equation (\Eq{conv-conc1}). Replacing the advective terms with a derivative along the characteristic direction, \Eq{conv-conc1} becomes 
\begin{align}\label{errorL1}
\psi_c \p{c}{\tau} + Gc = H,
\end{align}
whose finite difference approximation is given by
\begin{align}\label{errorL2}
\phi_i \frac{m_i^n -\bar{m}_i^{n-1}}{\Delta t} + G_i^n m_i^n = H_i^n.
\end{align}
The analysis of the component transport equation will be carried out in a manner similar to the water transport equation, by estimating the temporal discretization error, the approximation error of the source term, the interpolation error of the characteristics, and relevant auxiliary functions and their derivatives. To that effect, recall from \cref{error-def}, $\theta^n_i = c_i^n - m_i^n$. Using an analogue of Lemma \ref{lem:char-error} for the characteristic derivative of the polymer transport equation in \Eq{errorL1} and subtracting \Eq{errorL2} from the result we obtain
\begin{align}\label{conc-error1}
\phi_i \frac{\theta_i^n - (\bar{c}_i^{n-1} - \bar{m}_i^{n-1})}{\Delta t} + G_i^n \theta_i^n &= H(s_i^n) - H(w_i^n) + O\left(\nm*{\pderiv[2]{c}{\tau}}_\infty \Delta \tau\right) \nonumber\\
 &\leq \abs{H(s_i^n) - H(w_i^n)}+ M \Delta t
\end{align}
As before the source terms are estimated as
\begin{align}\label{errorL3}
\abs{H(s_i^n) - H(w_i^n)} \leq \abs{s_i^n -w_i^n} \nm*{\p{H}{s}}_\infty \leq M \abs{\zeta^n_i}
\end{align}
In the following, with slight abuse of notation, we suppress the superscript ``c'' from $\tilde{x}_i^{c,n}$ and $\bar{x}_i^{c,n}$ to denote the points on the characteristic curves of the polymer transport equation. Continuing with the analysis, we rewrite the numerator of the first term on the left side of \Eq{conc-error1}  as
\begin{align}\label{errorM}
\theta_i^n - (\bar{c}_i^{n-1} - \bar{m}_i^{n-1}) = (\theta^n_i - \bar{\theta} ^{n-1}_i) - \underbrace{(c^{n-1} (\bar{x_i}) - c^{n-1}(\tilde{x_i}))}_{\text{F}} - \underbrace{(1- \Ia)c^{n-1}(\tilde{x_i})}_{\text{G}}
\end{align}
The term G is estimated by the Peano kernel theorem, as was done in \Eq{pkestimate}.

\noindent{\underline{\bf F. Estimate of the term F}}: This estimate is carried out in a series of steps.
\begin{align}
&\abs{c^{n-1} (\bar{x_i}) - c^{n-1}(\tilde{x_i})} \leq \nm{c^{n-1}}_{1,\infty} \abs{\tilde{x_i} - \bar{x_i}} \nonumber\\
& \leq M\frac{\Delta t}{\phi_*}\underbrace{ \abs*{\frac{f}{s}(\bar{w}^n_i,m_i^{n-1})z_i^{n-1} - \frac{f}{s}(s_i^n,c_i^n)v_i^n}}_{\text{F-1}} \nonumber\\
& \qquad \qquad + M\frac{\Delta t}{\phi_*}\underbrace{\abs*{\frac{D}{s}(\bar{w}^n_i,m_i^{n-1})\pdx{w_i^n} - \frac{D}{s}(s_i^n,c_i^n)\pdx{s_i^n}}}_{\text{F-2}}\label{conc-error2}
\end{align}

\noindent{\underline {\bf F-1. Estimate of the term F-1}}:
\begin{align}
 \abs*{\frac{f}{s}(\bar{w}^n_i,m_i^{n-1})z_i^{n-1} - \frac{f}{s}(s_i^n,c_i^n)v_i^n} &\leq \abs*{\frac{f}{s}\bar{w}^n_i,m_i^{n-1})}\underbrace{\abs*{z_i^{n-1}-v_i^n}}_{\text{F-1-1}}\nonumber \\
  & \qquad  + \underbrace{\abs*{\frac{f}{s}(\bar{w}^n_i,m_i^{n-1})-\frac{f}{s}(s_i^n,c_i^n)}}_{\text{F-1-2}}\abs{v_i^n} \label{errorM-1}
\end{align}
Out of the two pieces F-1-1 and F-1-2 required to obtain an estimate of F-1, we have already estimated the term F-1-1 in \Eq{errorF} which we recall here: $\abs*{z_i^{n-1}-v_i^n} \leq M(h + \Delta t + \abs{\zeta_i^{n-1}}+\abs{\theta_i^{n-1}})  $. We next estimate the term F-1-2 in \Eq{errorF}.
\smallskip

\noindent{\underline{\bf F-1-2. Estimate of the term F-1-2 in \Eq{errorM-1}}}: 
\begin{align*}
& \abs*{\frac{f}{s}(\bar{w}^n_i,m_i^{n-1})-\frac{f}{s}(s_i^n,c_i^n)} \\
&\qquad \leq \abs*{\bar{w}_i^n - s_i^n}\,\nm*{\p{}{s}\left(\frac{f}{s}\right)}_{\infty} + \abs*{\theta_i^{n-1}}\,\nm*{\p{}{c}\left(\frac{f}{s}\right)}_{\infty} +  \Delta t\, \nm*{\pdt{c_i^n}}_{\infty}\nm*{\p{}{c}\left(\frac{f}{s}\right)}_{\infty}\\
&\qquad \leq M \left(\underbrace{\abs*{\bar{w}_i^n - s_i^n} }_{\text{F-1-2-a}} + \abs{\theta_i^{n-1}} +  \Delta t \right)
\end{align*}
\medskip

\noindent{\underline{\bf F-1-2-a. Estimate of the term F-1-2-a}}: 
\begin{align*}
\bar{w}_i^n - s_i^n &= \Ia s^n(\tilde{x}_i) - s^n(x_i)-\bar{\zeta}_i^n
= (\tilde{x}_i - x_i) \pdx{s^{n*}} - (1-\Ia)s^n(\tilde{x}_i) - \bar{\zeta}_i^n
\end{align*}
Therefore 
\begin{align}
\abs{\bar{w}_i^n - s_i^n} & \leq \abs{\tilde{x}_i - x_i}\, \nm{s^n}_{1,\infty} + Ch^2 + \abs{\bar{\zeta}_i^n} \quad (\text{Using the Peano-kernel theorem})\nonumber \\
&\leq M \Delta t \left\{\nm*{\frac{f}{s}}_{\infty}\, \nm{z^{n-1}}_{\infty} + \nm*{\frac{D}{s}}_\infty \left(\abs*{\pdx{s_i^n}}_\infty +\abs*{\pdx{\zeta_i^n}} \right) \right\}+ Ch^2 + \abs{\bar{\zeta}_i^n}\nonumber \\
&\leq M\Delta t \left\{ h+ C + \abs*{\pdx{\zeta_i^n}}\right\} + M h^2 +  \abs{\bar{\zeta}_i^n} \label{errorN1-2}
\end{align}
The last step of the above estimate in \Eq{errorN1-2} requires a bound on $\nm{z^{n-1}}_\infty$ which was also used while estimating E-3 and E-4 in \cref{errorK1,errorK2} respectively. Before further analysis, we prove this statement here. Note that, even though we prove the result for $\nm{z^n}_\infty$, it is true for any other time $t^n$ with $n\in (0,T)$. 
\begin{align}
z_i^{n} &= - K \lb(w_i^n,m_i^n )\pdx{r_i^n} = K  \lb(w_i^n,m_i^n ) \left[ \pdx{\pi_i^n} - \pdx{p_i^n}\right] \nonumber \\
\nm{z^n}_\infty &\leq \nm{K}_\infty \nm{\lb}_\infty \left(1+ \nm*{\pdx{\pi_i^n}}_\infty\right) \leq M ( 1 + \beta h); \quad (\beta \text{ is a constant}) \label{errorN1-3}
\end{align}
Using \Eq{errorN1-2} we obtain an estimate for F-1-2 as
\begin{align}
\abs*{\frac{f}{s}(\bar{w}^n_i,m_i^{n-1})-\frac{f}{s}(s_i^n,c_i^n)} \leq M\left( h^2 + \Delta t + h \Delta t + \abs*{\theta_i^{n-1}} + \abs*{\bar{\zeta_i^n}} + \Delta t \abs*{\pdx{\zeta^n}}\right)
\end{align}
Using these estimates of F-1-1 and F-1-2 in \Eq{errorM-1} we obtain an estimate of F-1 as 
\begin{align}\label{errorN1}
 \abs*{\frac{f}{s}(\bar{w}_i^n,m_i^{n-1})z_i^{n-1} - \frac{f}{s}(s_i^n,c_i^n)v_i^n} \leq M \left( h + \Delta t + h^2 + h\Delta t + \abs{\zeta_i^{n-1}}+ \abs{\theta_i^{n-1}} + \abs{\bar{\zeta_i^n}} + \Delta t \abs*{\pdx{\zeta^n}} \right)
\end{align}

\noindent{\underline{\bf F-2. Estimate of the term F-2 of \Eq{conc-error2}}}:
\begin{align}
&\abs*{\frac{D}{s}(\bar{w}^n_i,m_i^{n-1})\pdx{w_i^n} - \frac{D}{s}(s_i^n,c_i^n)\pdx{s_i^n}} \nonumber\\
&  \qquad \leq \left(\underbrace{\abs{\bar{w}_i^n-s_i^n}}_{\text{F-1-2-a}}\, \nm*{\p{}{s}\left(\frac{D}{s}\right)}_{\infty} + \abs{m_i^{n-1}-c_i^n}\, \nm*{\p{}{c}\left(\frac{D}{s}\right)}_{\infty} \right) \abs*{\pdx{s_i^n}} 
 + \nm*{\frac{D}{s}(\bar{w}_i^n,m_i^{n-1})}_{\infty} \abs*{\pdx{\zeta_i^n}}\label{errorN1-1}
\end{align}

\noindent Using the estimate for F-1-2-a given in \Eq{errorN1-2} in \Eq{errorN1-1}, we obtain 
\begin{align}
&\abs*{\frac{D}{s}(\bar{w}^n_i,m_i^{n-1})\pdx{w_i^n} - \frac{D}{s}(s_i^n,c_i^n)\pdx{s_i^n}} \leq M \left(\abs{\bar{w}_i^n - s_i^n} + \abs{\theta_i^{n-1}} + \Delta t + \abs*{\pdx{\zeta_i^n}} \right)\nonumber \\
& \quad \qquad \qquad  \leq M \left( \Delta t \left(C + h +\abs*{\pdx{\zeta_i^n}} \right) + h^2 +  \abs{\bar{\zeta}_i^n} + \abs{\theta_i^{n-1}} + \Delta t + \abs*{\pdx{\zeta_i^n}}  \right)\nonumber \\
& \quad \qquad \qquad \leq M\left(\Delta t + h^2+ h \Delta t + \abs{\bar{\zeta}_i^n} + \abs{\theta_i^{n-1}} + (1+ \Delta t)\abs*{\pdx{\zeta_i^n}} \right). \label{errorN2}
\end{align}
Using \Eq{errorN1} and \Eq{errorN2} in \Eq{conc-error2} we obtain the following estimate for the term F in \Eq{errorM}.
\begin{align}
 \abs{c^{n-1} (\bar{x_i}) - c^{n-1}(\tilde{x_i})} 
&\leq M \Delta t \left(h + \Delta t + h^2 + h\Delta t + \abs{\zeta_i^{n-1}}+ \abs{\theta_i^{n-1}} + \abs{\bar{\zeta_i^n}} + \Delta t \abs*{\pdx{\zeta^n}}\right) \nonumber \\
& \qquad + M \Delta t \left(\Delta t + h^2+ h \Delta t + \abs{\bar{\zeta}_i^n} + \abs{\theta_i^{n-1}} + (1+ \Delta t)\abs*{\pdx{\zeta_i^n}}\right) \nonumber\\
& \leq M \Delta t \left(  h + \Delta t +h^2 + h \Delta t + \abs{\zeta^{n-1}_i} + \abs{\theta^{n-1}_i} + \abs{\bar{\zeta}^n_i} + (1+ \Delta t) \abs*{\pdx{\zeta^n_i}}\right). \label{conc-error3}
\end{align} 
We test \Eq{conc-error1} against $\theta_i^n$ and using \Eq{errorL3} and \Eq{conc-error3}, we get
\begin{align*}
\langle &\phi_* \frac{\theta^n_i - \bar{\theta}^{n-1}_i}{\Delta t}, \theta_i^n \rangle + \langle \hat{M}\theta_i^n , \theta_i^n \rangle \\
&\leq \langle M \left( h + h^2 + \Delta t + h \Delta t + \abs{\zeta^{n-1}_i} + \abs{\theta^{n-1}_i} + \abs{\bar{\zeta}^n_i} + (1+ \Delta t) \abs*{\pdx{\zeta^n_i}}\right) , \theta_i^n \rangle \\ 
& \qquad + \langle M \left( \frac{h^2}{\Delta t} + \Delta t + \abs*{\zeta_i^n} \right), \theta_i^n \rangle.
\end{align*}
After some simplification, we get
\begin{align}\label{conc-error4}
 \phi_* (\nm{\theta^n}^2 &- \nm{\theta^{n-1}}^2)   \leq    \phi^* \nm*{\theta^n}^2 + (\phi^* - \phi_*) \nm*{\theta^{n-1}}^2 + \bar{M}\Delta t \left( h^4 + \Delta t^2 + h^2 +h^2 \Delta t^2 + \frac{h^4}{\Delta t^2}\right) \nonumber \\ 
& + M \Delta t \left( \nm{\theta^{n-1}}^2 + \nm{\zeta^{n-1}}^2 + \nm{\theta^{n}}^2 + \nm{\zeta^{n}}^2 + (1+ \Delta t ) (\abs{\zeta^n}^2_{1,2} + \nm{\theta^n}^2) \right).  
\end{align}
Adding \Eq{errorL} and \Eq{conc-error4}, summing over $1 \leq n \leq L$ and after some further simplification, we get
\begin{align}\label{conc-conv1}
M ( &\nm*{\zeta^L}^2 + \nm*{\theta^L}^2 ) + \bar{M} \Delta t  \sum_{n=1}^L \abs*{\zeta^n}^2_{1,2} \nonumber \\
&\leq M (1+ \Delta t + \Delta t^2) \sum_{n=0}^L \nm*{\theta^n}^2 + M \Delta t \sum_{n=0}^L \nm*{\zeta^n}^2 + MT \left(h^2 + \Delta t^2 + h^2 \Delta t^2 + h^4 + \frac{h^4}{\Delta t^2}\right) \nonumber \\
&+ M \Delta t \left( (1+ \Delta t) + (1+h) (\log{1/h})^{1/2} \right) \sum_{n=0}^L \abs{\zeta^n}^2_{1,2} 
\end{align}
where $T = L \Delta t$. Let $\rho = \left( (1+ \Delta t) + (1+h) (\log{1/h})^{1/2} \right)$ such that $\rho \rightarrow 0$. Then using the discrete Gr\"{o}nwall's inequality in \Eq{conc-conv1}, we get
\begin{align}\label{final-convergence}
&\nm*{\zeta^L}^2 + \nm*{\theta^L}^2 + \Delta t \sum_{n=0}^L \abs{\zeta^n}^2_{1,2} \leq M (h^2 + \Delta t ^2), \\
\text{ where } M &= M\left( \nm*{s}_{L^\infty(W^{3,\infty})}, \, \nm*{s}_{W^{1,\infty}(L^{\infty})}, \, \nm*{\frac{\partial^2 s}{\partial \tau^2}}_{L^\infty(L^{\infty})}, \, \nm*{c}_{L^\infty(W^{3,\infty})}, \, \nm*{c}_{W^{1,\infty}(L^{\infty})},  \right. \nonumber \\
& \qquad \qquad \left.  \nm*{\frac{\partial^2 c}{\partial \tau^2}}_{L^\infty(L^{\infty})}, \, \nm*{p}_{L^\infty(W^{1,\infty}) }, \, \nm*{p}_{W^{1,\infty}(W^{1,\infty})  }    \right) \nonumber.
\end{align} 
We anticipate an $L^2$ error of the order $O(h)$ and consequently, we assume that $\Delta t = O(h)$ as $h\rightarrow0$ which implies $\max(h^2 + \Delta t^2, h^4/\Delta t^2) = h^2 + \Delta t^2$.  In particular, it follows that 
\begin{align}\label{final-estimate}
\nm*{\zeta^L}_{L^2} \leq M h, \quad \nm*{\theta^L}_{L^2} \leq M h.
\end{align}
Note that $\Delta t = O(h)$ hypothesis is very reasonable since in the case of a one-dimensional parabolic equation the standard discretization can only be expected to yield an $O(h+\Delta t)$ estimate. Also, with a less stringent restriction like $\Delta t = O(h^\gamma)$ for some $\gamma<2$, we will have an $L^2$ error estimate of the order $O(h^{1-\gamma/2})$. The final error estimate is summarized in the following theorem. 
\begin{theorem}
Let $s$ and $c$ be the solutions of \Eq{conv-sat1} and \Eq{conv-conc1} respectively. Let $w$ and $m$ be the solutions of \Eq{conv-num2} and \Eq{conv-num11} respectively where $\bar{w}_i^{n-1}$ is given by \Eq{time-deriv1} and $\bar{m}_i^{n-1}$ is given by \Eq{time-deriv2}. Then, the errors $\zeta = s-w$ and $\theta = c-m$ satisfy the inequalities given in \Eq{final-estimate} and the convergence result given in \Eq{final-convergence}.  
\end{theorem}

\subsection{Extension to two dimensions in space}
Here we discuss how to extend the analysis of \Eq{conv-num11} and \Eq{conv-num2} to two spatial dimensions. The error estimates for the discretization of the characteristic derivatives and the capillary dissipation terms, obtained in Lemma~\ref{lem:char-error} and Lemma~\ref{lem:secdev-error} respectively, can be easily extended to two spatial dimensions without changing the order of the estimates. The various inequalities and tools used at various stages of the analysis like the Cauchy-Schwarz inequality, discrete Gr\"{o}nwall's inequality and the Taylor series have multidimensional analogues. The Peano kernel theorem can also be used in a similar manner for estimating the error introduced due to the bilinear interpolation required in the two-dimensional analysis. The $\nm*{\p{\pi}{x}}_\infty$ estimate obtained from the finite element solution of the elliptic pressure equation is also available in two or higher spatial dimensions \cite{HWW2010}. The $L^\infty$ estimate of a mesh function \cite{B1966}, that has been used to estimate the term E4, is also applicable for a two-dimensional domain. In the analysis of the one dimensional system, the spatial grid has been taken to be uniform with a fixed spatial grid size $h$. In the two dimensional system, the grid can be taken to be uniform in each spatial dimension with constant $h_x = \Delta x$ and $h_y = \Delta y$. The quasilinear treatment of the nonlinearity in the functions $f, D, \lb$ will allow us to obtain analogous estimates of the two-dimensional inner products involving these terms without affecting the convergence results. Hence, a comprehensive analytical study can be expected to yield an $O(h_x + h_y + \Delta t)$ error estimate for the two dimensional problem. However, a complete analysis of the two dimensional problem will also require careful re-estimation of some of the intermediate results, especially those involving the use of Sobolev inequalities on the higher order norms and semi-norms of the true solution and the auxiliary functions.    

\subsection{Extension to two component systems}
Here we discuss the possibility of extending this analysis to the case of two-component two-phase flows like surfactant-polymer flooding. Such a system has been studied recently in Daripa \& Dutta \cite{DD2017} which we present below.
\begin{subequations}
	\begin{align}
	& -\bgrad \cdot \left(\bK(\bx)\lb(s,c,\Gamma)\bgrad p\right) = q_a+q_o,  \label{eq:surfflood2-press} \\
	& \phi \p{s}{t} + \p{f_a}{s} \textbf{v}\cdot \bgrad s + \div \left(D \p{p_c}{s} \bgrad s \right) = g_s - \p{f_a}{c} \textbf{v}\cdot \bgrad c - \p{f_a}{\Gamma} \textbf{v}\cdot \bgrad \Gamma  - \div \left(D \p{p_c}{\Gamma} \bgrad \Gamma\right), \label{eq:surfflood2-1} \\ 
	& \phi\p{c}{t} + \left(\frac{f_a}{s} \textbf{v} + \frac{D}{s} \p{p_c}{s} \bgrad s  + \frac{D}{s} \p{p_c}{\Gamma} \bgrad \Gamma \right)\cdot\bgrad c = g_c, \label{eq:surfflood2-2} \\
	& \phi\p{\Gamma}{t} + \left(\frac{f_a}{s} \textbf{v} + \frac{D}{s} \p{p_c}{s} \bgrad s+ \frac{D}{s} \p{p_c}{\Gamma} \bgrad \Gamma\right) \cdot\bgrad \Gamma  = g_\Gamma, \label{eq:surfflood2-3} 
	\end{align}
\end{subequations}
where $D(s,c,\Gamma) =  \bK(\bx)\lb_o(s,\Gamma)f_a(s,c,\Gamma)$, $\Gamma$ is the surfactant concentration and the source terms $q_a$, $q_o$, $g_s$, $g_c$, $g_\Gamma$ are defined similar to the one-component flow model. As seen from the above model, the transport equations for polymer and surfactant have a similar structure. Hence the surfactant transport equation can be analyzed in a similar fashion to obtain error estimates for this two-component two-phase flow system. However, such an exercise also poses certain challenges. The functions $p_c$ and $D,f,\lb, \lb_a, \lb_o$ are not always dependent on all three components $s,c,\Gamma$. For instance, the capillary pressure $p_c = p_c(s,\Gamma)$ is only affected by changes in water saturation $s$ and surfactant concentration $\Gamma$, whereas the fractional flow functions $f_a = f_a(s,c,\Gamma)$ and $f_o = f_o(s,c,\Gamma)$, depend on all three. Similarly, $\lb_a = \lb_a(s,c,\Gamma)$ but $\lb_o = \lb_o(s,\Gamma)$. This means that the estimates are not always symmetric with respect to the two transport variables $c$ and $\Gamma$. Hence, an analogous error estimate for the two-component system is difficult to obtain as a direct extension of the one-component system and it needs further non-trivial analysis. However, due to the similarity in the structure of the transport equations for $c$ and $\Gamma$ and because the numerical method for the two-component system is a direct extension of the one-component system, we anticipate an equivalent $L^2$ estimate of the order $O(h+\Delta t)$ even for the error in surfactant concentration.    

\section{Numerical results}\label{sec:results}
In \cite{DD2017}, an exact solution for the two-dimensional, immiscible, two-phase flow problem has been constructed and used for numerical verification of the convergence and the order of accuracy of the numerical method. The $L^2$ and $L^\infty$ error norms and the respective orders of accuracy with spatial grid refinement have been presented there for the water saturation $s$, the pressure $p$ and the velocity $\bv$. The $L^2$ and $L^\infty$ error norms for $s$ with time step refinement at different fixed spatial grid sizes have been also presented along with the corresponding convergence rates. The $L^2$ error in saturation has been shown to be of the order $O(h)$. This is consistent with the estimate obtained from our one-dimensional analysis presented in this paper and, as discussed above, is expected to be true in two-dimensional case. In \cite{DD2017}, the $O(h)$ error in the $L^\infty$ norm of the gradient of pressure has also been observed in the numerical results obtained with an exact solution. Additionally, the numerical water saturation profiles at various spatial resolutions have been found to compare favorably with the saturation profiles for the exact solution, thus providing support for the convergence of the numerical method.

Several numerical experiments have been carried out using this numerical method to simulate practical two-phase flow problems, both with and without components, that arise in the context of chemical enhanced oil recovery. Mainly two different types of computational domains have been employed for these simulations - a quarter of a five-spot geometry (or radial flow) that mimics oil reservoir conditions near the location of the physical sources, and a rectilinear geometry that mimics the flow conditions far from the location of the sources. A variety of different types of heterogeneous permeability fields have also been used. These include rectangular inclusions, channelized domains, a multiscale permeability field generated using a stationary, isotropic, fractal Gaussian field and sections of the SPE10 permeability field. These simulation results have been used to qualitatively validate the numerical method by comparing with results from existing literature. Also, several comparison studies have been performed between different combinations of single or multi-component, two phase flows which demonstrate the capability of the method to capture the intricate details of the flow characteristics and produce numerical results that are consistent with expectations based on physics. The reader is directed to Daripa \& Dutta \cite{DD2017} for further details.    

In this section, we present numerical results obtained from solving the two-phase single component system of equations (polymer flooding), given by \cref{paper2-glbfinal,paper2-sateq,paper2-conceq}, subject to realistic initial and boundary data. These are intended to support the error calculations with respect to an exact solution that were presented in \cite{DD2017} and provide further quantitative evidence about the accuracy and the order of convergence of the method, even in the case of practical numerical simulations. In \cref{sec:result2}, the input data is given in \cref{table:paper2-input_par}.
The numerical errors are measured in the discrete norms.
\begin{subequations}\label{error-norms}
	\begin{align}
	e_{s,max} &= \displaystyle\max_{ij}{\abs{s(\bx_{ij})-w_{ij}}} \equiv \nm{s - w}_{L^\infty}, \\
	e_{s,2} &= \sqrt{\displaystyle\sum_{ij}\abs{s(\bx_{ij})-w_{ij}}^2 \Delta x \Delta y} \equiv \nm{s-w}_{L^2}. 
	\end{align}
\end{subequations}
Here, $w_{ij}$ is the numerical solution $w$ evaluated at the grid point $(x_i,y_j) = \bx_{ij}$ whereas $s(\bx_{ij})$ is the finest grid numerical solution in \cref{sec:result2}. The errors for the pressure and the velocity are computed in a similar fashion. The order of accuracy is computed using the formula $\log_2{(e_{\alpha}(h)/e_{\alpha}(h/2))}$ ($\alpha = 2,\infty$).  

\subsection{Two-dimensional polymer flood problem}\label{sec:result2}
We perform simulation of polymer flooding on a quarter five-spot homogeneous geometry $\Omega = [0,1]^2$ with absolute permeability, $\bK = 1$ and input parameters listed in \Cref{table:paper2-input_par}. The transport source terms in \cref{paper2-sateq,paper2-conceq} for a quarter-five spot flow geometry are taken as
\begin{align*}
g_s =\begin{cases}
(1-f_a)Q\\
0 
\end{cases} 
g_c = \begin{cases}
(c^i - c) Q/s \\
0 
\end{cases}
\text{ at } \bx = \begin{cases}
\bx^i \equiv (0,0) \quad \text{ (source)}\\
\Omega \setminus \{\bx^i\} \quad \text{ (elsewhere)}
\end{cases}.
\end{align*} 
The source terms for the pressure \cref{paper2-glbfinal} are taken as
\begin{align*}
&q_a =\begin{cases}
Q \\
-(\lb_a/\lb) \, Q \, ; \\
0 
\end{cases} 
q_o = \begin{cases}
0   \\
-(\lb_o/\lb) \, Q  \\
0 
\end{cases} 
\text{ at} \,
\begin{cases}
\bx^i = (0,0) \qquad \qquad \qquad \text{ (Source)} \\
\bx^p = (1,1) \qquad \qquad \qquad \text{ (Sink)} \\
\bx \in \Omega\setminus\{(0,0) \cup (1,1)\} \, \text{ (Elsewhere)}
\end{cases}.
\end{align*} 
\renewcommand{\arraystretch}{1}
\begin{table}[h!]
	\caption{Simulation input data}
	\centering
	\begin{tabular}{||l  c c||} 
		\hline
		Model parameter       & Symbol        & Value  \\ [0.5ex] 
		\hline\hline
		Spatial grid size     & $h \times k$  & variable \\
		Porosity              & $\phi$        & 1  \\
		Initial resident water saturation & $s^{\sigma0}_0$ & $0.21$ \\	
		Oil viscosity         & $\mu_o$       &  12.6 \\
		Pure water viscosity  & $\mu_w$       & 6.3  \\
		Residual aqueous phase saturation & $s_{ra}$  & 0.1  \\
		Residual oleic phase saturation   & $s_{ro}$  & 0.2  \\
		Parameters of capillary pressure relation [\cref{paper2-eq:cap}] & $\alpha_0$, $m$ & $0.125$, $2/3$ \\
		Concentration of polymer in injected fluid & $c_0$ & 0.1 \\
		Injection rate        & $Q$   & 200 \\ 
		Time step size        & $\Delta t$    & $1/50$ \\ [.5ex] 
		\hline
	\end{tabular}
	\label{table:paper2-input_par}
\end{table}
We compute the $L^\infty$ and $L^2$ error norms of the numerical solutions for the saturation on a sequence of uniformly refined meshes $h=1/8, 1/16, 1/32, 1/64$ using \cref{error-norms}, but with $s(\bx_{ij})$ representing the solution on the finest grid size $h=1/128$. A similar procedure is applied to estimate the error norms and the order of accuracy for the pressure and the velocity. The numerical errors and the order of accuracy are presented in \Cref{table:paper2-error1}. 
In \Cref{table:paper2-error2} we present the numerical errors and convergence rates with respect to time step size refinement $\Delta t = 1/20, \ldots , 1/160$ by keeping the spatial grid size fixed at three different levels $h= 1/16, 1/32$ and $1/64$ for the quarter five-spot flooding problem. The error calculations for both the tables have been performed at the time of water breakthrough which is given by the time at which the water saturation at $\bx^p$ reaches a chosen threshold value. We observe (see \Cref{table:paper2-error1}) the following approximate orders of accuracy in space. 
\begin{align*}
\nm{s-w}_{L^2} = O(h), \quad \nm{p-r}_{L^2} = O(h^2) \quad \& \quad \nm{\bv-\bm{z}}_{L^2} = O(h^2)
\end{align*} 
The $O(h)$ error in the $L^2$ norm for saturation $s$ directly matches with the estimate obtained from our one-dimensional analysis in \cref{final-estimate}. The $O(h)$ error in the gradient of the pressure $p$ (as seen in \cref{errorD}) is also observed in the $L^\infty$ norm $\nm{\bv - \bm{z}}_{L^\infty}$ for the velocity in \Cref{table:paper2-error1}. Moreover, the orders of accuracy in \Cref{table:paper2-error1} for the $L^2$ and $L^\infty$ errors of all the three variables $s$, $p$ and $\bv$ are consistent with the orders of accuracy obtained using an exact solution (see Table 1 in \cite{DD2017}).
  
The order of accuracy in the $L^\infty$ norm of the error in saturation, as presented in the upper part of the last column of \Cref{table:paper2-error1} can be seen to reduce significantly with reduction in spatial grid size. 
This is the because the saturation and its $L^\infty$ norm are both highly sensitive to minor changes in the flow and domain parameters, especially the ones whose $L^\infty$ bounds enter the generic coefficient $M$ used in \Eq{final-convergence} and in various other intermediate estimates obtained in \Cref{sec:analysis}. Hence the $L^\infty$ error norms of saturation in \Cref{table:paper2-error1} are at least an order higher than the corresponding norms for pressure and velocity, both of which are less sensitive to minor changes in the parameter space. To overcome this, a finer spatial grid size and time step size (data not shown here) need to be adopted for the numerical solution of the transport equations.

\Cref{table:paper2-error2} shows the $L^2$ error in saturation and the rate of convergence with respect to time. The results confirm that approximately a first order convergence rate in time can be obtained using this method. This compares favorably with results obtained using an exact solution (see Table 2 in \cite{DD2017}) and also with the convergence rate expected from the first order time discretization scheme used in the numerical formulation. We believe that with higher order time-stepping methods, the method will be able to preserve the accuracy and the expected second or third order convergence rates.

\renewcommand{\arraystretch}{1.3}
\begin{table}[ht!]
	\caption{Error and order for saturation, pressure and velocity at water breakthrough of a quarter five-spot polymer flooding simulation}
	\begin{center}
		\begin{tabular}{l c c c c c}
			\hline
			& $h$ & $\nm{s-w}_{L^2}$ & ~Order~ & $\nm{s-w}_{L^\infty}$ & ~Order~ \\
			\hline
			\hline
			\multirow{4}{*}{\rotatebox[origin=c]{90}{Saturation}} ~
			& $1/8$  & $4.39$e$-$3 & $-$     & $3.64$e$-$2 & $-$ \\ 
			& $1/16$ & $1.85$e$-$3 & $1.247$ & $1.76$e$-$2 & $1.048$ \\
			& $1/32$ & $7.84$e$-$4 & $1.239$ & $1.08$e$-$2 & $0.704$ \\
			& $1/64$ & $3.22$e$-$4 & $1.284$ & $6.86$e$-$3 & $0.656$ \\
			\hline
			& $h$ & $\nm{p-r}_{L^2}$ & ~Order~ & $\nm{p-r}_{L^\infty}$ & ~Order~ \\
			\hline
			\hline
			\multirow{4}{*}{\rotatebox[origin=c]{90}{Pressure}} ~ 
			& $1/8$  & $4.12$e$-$3 & $-$     & $1.61$e$-$3 & $-$ \\
			& $1/16$ & $9.30$e$-$4 & $2.147$ & $4.75$e$-$4 & $1.761$ \\
			& $1/32$ & $2.10$e$-$4 & $2.147$ & $1.35$e$-$4 & $1.815$ \\
			& $1/64$ & $3.85$e$-$5 & $2.448$ & $3.10$e$-$5 & $2.123$ \\
			\hline
			& $h$ & $\nm{\bv-{\bf z}}_{L^2}$ & ~Order~ & $\nm{\bv-{\bf z}}_{L^\infty}$ & ~Order~ \\
			\hline
			\hline
			\multirow{4}{*}{\rotatebox[origin=c]{90}{Velocity}} ~ 
			& $1/8$  & $1.18$e$-$3 & $-$     & $6.94$e$-$3 & $-$ \\
			& $1/16$ & $2.94$e$-$4 & $2.005$ & $2.98$e$-$3 & $1.220$ \\
			& $1/32$ & $7.46$e$-$5 & $1.979$ & $1.39$e$-$3 & $1.110$ \\
			& $1/64$ & $1.96$e$-$5 & $1.928$ & $6.89$e$-$4 & $1.002$ \\
			\hline
			\hline
		\end{tabular}
		\label{table:paper2-error1}
	\end{center}
\end{table}

\renewcommand{\arraystretch}{1.4}
\begin{table}[ht!]
	\caption{Error and rates for saturation with time step refinement at water breakthrough of a quarter five-spot polymer flooding simulation.}
	\begin{center}
		\begin{tabular}{c | c c | c c | c c}
			\hline
			\multirow{2}{*}{$\Delta t$} & \multicolumn{2}{c|}{$h = 1/16$} & \multicolumn{2}{c}{$h = 1/32$} & \multicolumn{2}{c}{$h = 1/64$} \\
						 & $\nm{s-w}_{L^2}$ & ~Rate~ & $\nm{s-w}_{L^2}$ & ~Rate~  & $\nm{s-w}_{L^2}$ & ~Rate~ \\					   
			\hline
			\hline
				$1/20$   & $9.34$e$-$3     & $-$     & $8.93$e$-$3      & $-$     & $6.50$e$-$3      & $-$   \\
				$1/40$   & $4.36$e$-$3     & $1.100$ & $4.69$e$-$3      & $0.923$ & $3.48$e$-$3      & $0.901$ \\
				$1/80$   & $2.27$e$-$3     & $0.941$ & $2.51$e$-$3      & $0.899$ & $1.94$e$-$3      & $0.846$ \\
				$1/160$  & $1.25$e$-$3     & $0.867$ & $1.48$e$-$3      & $0.762$ & $1.17$e$-$3      & $0.722$ \\
			\hline
			\hline
		\end{tabular}
		\label{table:paper2-error2}
	\end{center}
\end{table}

\section{Conclusions}\label{sec:conclusions}
In Daripa \& Dutta~\cite{DD2017}, a hybrid numerical method was proposed for solving a two-phase two-component flow problem in porous media and was applied to successfully solve some relevant two-dimensional problems. The hybrid method uses a non-traditional discontinuous finite element method for solving the elliptic equation and a time-implicit finite difference method in combination with the modified method of characteristics for solving the transport equations. Numerical results presented there are in excellent agreement with the physics of flow as well as with exact solutions when available, and are also demonstrated to converge under mesh refinement. In this paper, we perform numerical analysis of the method to establish convergence by considering a reduced system, namely two-phase, one-component porous media flow in one-dimension. The novelty in the paper is the consideration of the single component system of transport equations which significantly complicates the analysis previously performed by others \cite{DR1982,D1983} without any component. Basic ideas of the proof can be extended to two-dimensions and to two-component systems as discussed in this article but needless to say, a complete study will be even more technically involved and beyond the scope of this work.
	
In the analysis presented here, the convergence behavior of the MMOC-FD part of this hybrid numerical method has been studied. An optimal order $O(h)$ error has been assumed for the pressure gradient obtained by the finite element part, which had been numerically validated in \cite{HWW2010}. Using this result, $L^2$ error estimates of the wetting phase saturation $s$ and the component concentration $c$ have been computed. Numerical experiments have been performed to simulate two-phase one-component flow in a quarter five-spot geometry. The $L^2$ error norm of the saturation $s$ and the $L^\infty$ error norm of the velocity obtained numerically have been used to verify the theoretical error estimates. These $L^2$ and $L^\infty$ error norms have been also used to demonstrate the numerical convergence of the method as well as the order of accuracy with spatial grid refinement and the convergence rates with respect to the time step refinement.

\section*{Acknowledgments}
The research reported in this paper has been supported in part by the U.S. National Science Foundation grant DMS-1522782, and by an appointment of Sourav Dutta to the Postgraduate Research Participation Program at the U.S.
Army Engineer Research and Development Center, Coastal and Hydraulics Laboratory (ERDC-CHL) administered by the Oak Ridge Institute for Science and Education through an interagency agreement between the U.S. Department of Energy and ERDC. Some of the numerical simulations have been performed using high-performance research computing resources provided by Texas A\&M University (\url{http://hprc.tamu.edu}). Permission was granted by the Chief of Engineers to publish this information. 

\normalem
\bibliographystyle{unsrt}
\bibliography{ref-abb}

\begin{thebibliography}{10}

\bibitem{DD2017}
P.~Daripa and S.~Dutta.
\newblock Modeling and simulation of surfactant-polymer flooding using a new
  hybrid method.
\newblock {\em J. Comput. Phys.}, 335:249--282, 2017.

\bibitem{DGLM1988-1}
P.~Daripa, J.~Glimm, B.~Lindquist, M.~Maesumi, and O.~McBryan.
\newblock On the simulation of heterogeneous petroleum reservoirs.
\newblock In {\em Numerical Simulation in Oil Recovery, IMA Vol. Math. Appl.
  11}, pages 89--103, New York, NY, 1988. Springer.

\bibitem{DGLM1988-2}
P.~Daripa, J.~Glimm, B.~Lindquist, and O.~McBryan.
\newblock Polymer floods: A case study of nonlinear wave analysis and of
  instability control in tertiary oil recovery.
\newblock {\em SIAM J. Appl. Math.}, 48:353--373, 1988.

\bibitem{DR1982}
J.~{Douglas Jr.} and T.~F. Russell.
\newblock Numerical methods for convection-dominated diffusion problems based
  on combining the method of characteristics with finite element or finite
  difference procedures.
\newblock {\em SIAM J. Numer. Anal.}, 19(5):871--885, 1982.

\bibitem{D1983}
J.~{Douglas Jr.}
\newblock Finite difference methods for two-phase incompressible flow in porous
  media.
\newblock {\em SIAM J. Numer. Anal.}, 20(4):681--696, 1983.

\bibitem{EW1980}
R.~E. Ewing and M.~F. Wheeler.
\newblock {Galerkin methods for miscible displacement problems in porous
  media}.
\newblock {\em SIAM J. Numer. Anal.}, 17(3):351--365, 1980.

\bibitem{R1985}
T.~F. Russell.
\newblock {Time stepping along characteristics with incomplete iteration for a
  Galerkin approximation of miscible displacement in porous media}.
\newblock {\em SIAM J. Numer. Anal.}, 22(5):970--1013, 1985.

\bibitem{D1993}
L.~J. Durlofsky.
\newblock A triangle based mixed finite element-finite volume technique for
  modeling two phase flow through porous media.
\newblock {\em J. Comput. Phys.}, 105(2):252--266, 1993.

\bibitem{CPV2014}
C.~Canc{\`{e}}s, I.~S. Pop, and M.~Vohral{\'{i}}k.
\newblock {An a-posteriori error estimate for vertex-centered finite volume
  discretizations of immiscible incompressible two-phase flow}.
\newblock {\em Math. Comput.}, 83(285):153--188, 2014.

\bibitem{EHM2003}
R.~Eymard, R.~Herbin, and A.~Michel.
\newblock {Mathematical study of a petroleum-engineering scheme}.
\newblock {\em ESAIM Math. Model. Numer. Anal.}, 37(6):937--972, 2003.

\bibitem{CCSS2001}
P.~Castillo, B.~Cockburn, D.~Sch{\"{o}}tzau, and C.~Schwab.
\newblock {Optimal a priori error estimates for the $hp$-version of the local
  discontinuous Galerkin method for convection--diffusion problems}.
\newblock {\em Math. Comput.}, 71(238):455--479, 2001.

\bibitem{KP2016}
S.~Karpinski and I.~S. Pop.
\newblock {Analysis of an interior penalty discontinuous Galerkin scheme for
  two phase flow in porous media with dynamic capillary effects}.
\newblock {\em Numer. Math.}, 136(1):1--38, 2016.

\bibitem{NSM2004}
D.~Nayagum, G.~Sch\"{a}fer, and R.~Mos\'{e}.
\newblock Modelling two-phase incompressible flow in porous media using mixed
  hybrid and discontinuous finite elements.
\newblock {\em Comput. Geosci.}, 8(1):49--73, 2004.

\bibitem{DHP1999}
J.~{Douglas Jr.}, C.-S. Huang, and F.~Pereira.
\newblock The modified method of characteristics with adjusted advection.
\newblock {\em Numer. Math.}, 83(3):353--369, 1999.

\bibitem{CRHE1990}
M.~A. Celia, T.~F. Russell, I.~Herrera, and R.~E. Ewing.
\newblock An {Eulerian-Lagrangian} localized adjoint method for the
  advection-diffusion equation.
\newblock {\em Adv. Water Resour.}, 13(4):187--206, 1990.

\bibitem{AW1995}
T.~Arbogast and M.~F. Wheeler.
\newblock A characteristics-mixed finite element method for advection-dominated
  transport problems.
\newblock {\em SIAM J. Numer. Anal.}, 32(2):404--424, 1995.

\bibitem{W2000}
H.~Wang.
\newblock An optimal-order error estimate for an {ELLAM} scheme for
  two-dimensional linear advection-diffusion equations.
\newblock {\em SIAM J. Numer. Anal.}, 37(4):1338--1368, 2000.

\bibitem{WW2007}
H.~Wang and K.~Wang.
\newblock {Uniform estimates for Eulerian-Lagrangian methods for singularly
  perturbed time-dependent problems}.
\newblock {\em SIAM J. Numer. Anal.}, 45(3):1305--1329, 2007.

\bibitem{BK2012}
M.~Bause and P.~Knabner.
\newblock {Uniform error analysis for Lagrange-Galerkin approximations of
  convection-dominated problems}.
\newblock {\em SIAM J. Numer. Anal.}, 39(6):1954--1984, 2012.

\bibitem{MM1936}
M.~Muskat and M.~W. Meres.
\newblock {The flow of heterogeneous fluids through porous media}.
\newblock {\em J. Appl. Phys.}, 7(921):346--363, 1936.

\bibitem{CJ1986}
G.~Chavent and J.~Jaffr{\'{e}}.
\newblock {\em {Mathematical Models and Finite Elements for Reservoir
  Simulation}}.
\newblock Studies in Mathematics and its applications, Amsterdam, North
  Holland, 1st edition, 1986.

\bibitem{AJK2014}
B.~Amaziane, M.~Jurak, and A.~{{\v{Z}}galji{\'{c}} Keko}.
\newblock {Modeling compositional compressible two-phase flow in porous media
  by the concept of the global pressure}.
\newblock {\em Comput. Geosc.}, 18(3-4):297--309, 2014.

\bibitem{E1991}
R.~E. Ewing.
\newblock {Simulation of multiphase flows in porous media}.
\newblock {\em Transp. Porous Media}, 6(5-6):479--499, 1991.

\bibitem{CHM2006}
Z.~Chen, G.~Huan, and Y.~Ma.
\newblock {Well Modeling}.
\newblock In {\em Computational Methods for Multiphase Flows in Porous Media},
  pages 445--475. {SIAM}, 2006.

\bibitem{DFP1997}
J.~Douglas~Jr., F.~Furtado, and F.~Pereira.
\newblock On the numerical simulation of waterflooding of heterogeneous
  petroleum reservoirs.
\newblock {\em Comput. Geosci.}, 1:155--190, 1997.

\bibitem{C1986}
A.~T. Corey.
\newblock {\em Mechanics of Immiscible Fluids in Porous Media}.
\newblock Water Resources Publications, Littlton, Colorado, 1986.

\bibitem{vG1980}
M.~{Th}. van Genuchten.
\newblock A closed form equation for predicting the hydraulic conductivity of
  unsaturated soils.
\newblock {\em Soil Sci. Soc.}, 44:892--898, 1980.

\bibitem{PLK1987}
J.~C. Parker, R.~J. Lenhard, and T.~Kuppusamy.
\newblock A parametric model for constitutive properties governing multiphase
  flow in porous media.
\newblock {\em Water Resour. Res.}, 23(4):618--624, 1987.

\bibitem{GLGV2010}
B.~Ghanbarian-Alavijeh, A.~Liaghat, G.-H. Huang, and M.~{Th}. {v}an Genuchten.
\newblock Estimation of the van {G}enuchten soil water retention properties
  from soil textural data.
\newblock {\em Pedosphere}, 20(4):456--465, 2010.

\bibitem{HWW2010}
S.~Hou, W.~Wang, and L.~Wang.
\newblock Numerical method for solving matrix coefficient elliptic equation
  with sharp-edged interfaces.
\newblock {\em J. Comput. Phys.}, 229:7162--7179, 2010.

\bibitem{LLW2003}
Z.~Li, T.~Lin, and X.~Wu.
\newblock New {C}artesian grid methods for interface problems using the finite
  element formulation.
\newblock {\em Numer. Math.}, 96:61--98, 2003.

\bibitem{HL2005}
S.~Hou and X.~D. Liu.
\newblock {A numerical method for solving variable coefficient elliptic
  equation with interfaces}.
\newblock {\em J. Comput. Phys.}, 202(2):411--445, 2005.

\bibitem{B1966}
J.~H. Bramble.
\newblock Second order finite difference analogue of the first biharmonic
  boundary value problem.
\newblock {\em Numer. Math.}, 9:236--249, 1966.

\end{thebibliography}
\end{document}